\documentclass[1 [leqno,11pt]{amsart}
\usepackage{amssymb, amsmath}
\allowdisplaybreaks[4]
\usepackage{color}

\let\pa=\partial
\let\p=\partial
\let\ve=\varepsilon
\let\e=\varepsilon
\let\al=\alpha
\let\f=\frac
\let\D=\Delta

\let\wt=\widetilde
\let\oo=\infty

\def\no{\noindent}
\def\na{\nabla}
\def\th{\theta}
\def\dive{\mathop{\rm div}\nolimits}
\def\diveh{\mathop{\rm div}_\h\nolimits}

\def\cA{{\mathcal A}}
\def\cB{{\mathcal B}}

\def\cG{{\mathcal G}}
\def\cH{{\mathcal H}}

\def\Om{\Omega}

\def\eqdef{\buildrel\hbox{\footnotesize def}\over =}
\def\eqdefa{\buildrel\hbox{\footnotesize def}\over =}
\def\N{\mathop{\mathbb N\kern 0pt}\nolimits}
\def\Q{\mathop{\mathbb Q\kern 0pt}\nolimits}
\def\R{\mathop{\mathbb R\kern 0pt}\nolimits}
\def\Z{\mathop{\mathbb Z\kern 0pt}\nolimits}

\newcommand{\beq}{\begin{equation}}
\newcommand{\eeq}{\end{equation}}
\newcommand{\ben}{\begin{eqnarray}}
\newcommand{\een}{\end{eqnarray}}
\newcommand{\beno}{\begin{eqnarray*}}
\newcommand{\eeno}{\end{eqnarray*}}
\newcommand{\andf}{\quad\hbox{and}\quad}
\newcommand{\with}{\quad\hbox{with}\quad}

\newtheorem{thm}{Theorem}[section]
\newtheorem{lem}{Lemma}[section]
\newtheorem{rmk}{Remark}[section]
\newtheorem{cor}{Corollary}[section]
\newtheorem{prop}{Proposition}[section]

\newcommand{\vv}[1]{\boldsymbol{#1}}

\def\ur{u^r}
\def\ut{u^\theta}
\def\uz{u^z}
\def\wr{w^r}
\def\wwt{w^\theta}
\def\wz{w^z}
\def\vpr{\varpi^r}
\def\vpt{\varpi^\theta}
\def\vpz{\varpi^z}
\def\vr{v^r}
\def\vt{v^\theta}
\def\vz{v^z}
\def\fr{f^r}
\def\ft{f^\theta}
\def\fz{f^z}
\def\gr{g^r}
\def\gt{g^\theta}
\def\gz{g^z}

\def\h{{\rm h}}
\def\v{{\rm v}}
\def\nablah{\nabla_\h}

\numberwithin{equation}{section}

\begin{document}
\title[Global stability of large Fourier mode for 3-D N-S]
{Global stability of large Fourier mode for 3-D anisotropic Navier-Stokes equations in cylindrical domain}

\author[N. Liu]{Ning Liu}
\address[N. Liu]
{Academy of Mathematics $\&$ Systems Science, Chinese Academy of Sciences, Beijing 100190, CHINA.}
\email{liuning16@mails.ucas.ac.cn}

\author[Y. Liu]{Yanlin Liu}
\address[Y. Liu]{School of Mathematical Sciences,
Laboratory of Mathematics and Complex Systems,
MOE, Beijing Normal University, 100875 Beijing, China.}
\email{liuyanlin@bnu.edu.cn}

\author[P. Zhang]{Ping Zhang}
\address[P. Zhang]{Academy of Mathematics $\&$ Systems Science
	and Hua Loo-Keng Center for Mathematical Sciences, Chinese Academy of
	Sciences, Beijing 100190, CHINA, and School of Mathematical Sciences,
	University of Chinese Academy of Sciences, Beijing 100049, China.} \email{zp@amss.ac.cn}

\date{\today}

\maketitle

\begin{abstract} In this paper, we first establish the global existence and stability of solutions to 3-D classical Navier-Stokes equations
$(NS)$ in an
infinite cylindrical domain with large Fourier mode initial data. Then we extend similar result for 3-D anisotropic Navier-Stokes equations
$(ANS).$ We remark that due to the loss of vertical viscosity  in $(ANS),$ the construction of the energy functionals
for $(ANS)$ is much more subtle than that of $(NS).$
Compared with our previous paper \cite{LZ6} for $(NS)$, we improve the polynomial decay in $k$ for the Fourier coefficients of the solution
 to be exponential decay in
$k$ here.
\end{abstract}




\setcounter{equation}{0}
\section{Introduction}

 In this paper, we  investigate the global well-posedness and stability of the
 following   3-D incompressible (anisotropic) Navier-Stokes equations
  in the infinite cylindrical domain $\Omega\eqdef\left\{x=(x_1,x_2,x_3)\in \R^3: x_1^2+x_2^2< 1,\ x_3\in\R\  \right\}$ with large Fourier mode initial data:
\beq\label{NS}
\left\{
\begin{aligned}
&\p_t u+ u\cdot\nabla u-(\Delta_\h+\nu^2\pa_3^2) u+\nabla P= 0,
\qquad(t,x)\in\R^+\times\Omega,\\
&\dive u=0,\\
&u|_{\pa\Omega}=0,
\quad u|_{t=0} =u_{\rm in},
\end{aligned}
\right.
\eeq
where $\Delta_\h\eqdefa\pa_{x_1}^2+\pa_{x_2}^2$ denotes the horizontal Laplacian operator, $u$ and $P$
 stands for the velocity and the scalar pressure function of a viscous incompressible fluid respectively,  and $\nu\geq0$ designates the vertical viscous coefficient.

Equations of this type appear in geophysical fluid dynamics (see \cite{CDGGbook, Pedlovsky} for instance),
where meteorologists often simulate turbulent diffusion by putting
viscosity of the form: $-\mu_\h\Delta_\h-\mu_3\pa_3^2$,
and the empirical constant $\mu_3$ is usually much smaller than  the empirical constant $\mu_\h$.
We refer to the book of Pedlovsky \cite{Pedlovsky},
Chapter $4$ for a  complete discussion about this model. Here for simplicity, we just take $\mu_\h=1$ and $\mu_3=\nu,$
so that the system \eqref{NS} restores the classical (resp. anisotropic) Navier-Stokes equations, in short $(NS)$ (resp. $(ANS)$), in the case $\nu=1$ (resp. $\nu=0$).

\smallskip

For $(NS)$, Leray proved in the seminal paper \cite{lerayns} that for any solenoidal vector field $u_{\rm in}$ in $H^1(\R^3)$,
 the system $(NS)$ has a unique local strong solution $u\in C([0,T]; H^1(\R^3))\cap L^2(]0,T[,\dot H^2(\R^3))$, which
  becomes  analytic \cite{FT,HS}, for any time $t\in]0,T]$. Later, Kiselev and Ladyzhenskaya \cite{KL57} established the same results in bounded domain. However, whether or not this local smooth solution blows up in finite time is still a fundamental open problem in the field of mathematical fluid mechanics (see \cite{CC88, Fe06, Tao} for instance).

Yet for initial data with some special structures, there are affirmative answers to the above question. The most widely studied case is perhaps the axially symmetric solutions of $(NS),$ which reads
$$
u(t,x)=u^r(t,r,z)\vv e_r+u^\theta(t,r,z)\vv e_\theta+u^z(t,r,z)\vv e_z,
$$
where $(r,\theta,z)$ denotes the cylindrical coordinates in $\R^3$
so that $x=(r\cos\theta,r\sin\theta,z)$, and
$$\vv e_r\eqdefa (\cos\theta,\sin\theta,0),\ \vv e_\theta\eqdefa (-\sin\theta,\cos\theta,0),\
\vv e_z\eqdefa (0,0,1),\ r\eqdefa \sqrt{x_1^2+x_2^2}.$$
We call such a solution $u$ of $(NS)$ to be
axially symmetric without swirl  if in addition $\ut=0$.

For the axially symmetric initial data without swirl, Ladyzhenskaya \cite{La} and independently Ukhovskii and Yudovich \cite{UY} proved the global existence of strong solution to $(NS)$, which keeps this symmetric property. Later Leonardi et al. \cite{LMNP} gave a simplified proof to this result.
 Nevertheless, the global existence of strong solutions to $(NS)$ with general axially symmetric data is still open except for the case when $u^\th_{\rm in}$ is sufficiently small. The strategy of the proof is based on a perturbation argument for the  swirl free case.
There are numerous works concerning this topic,
here we only list the references \cite{CL02, Yau1, Yau2, KNSS, Lei, Liu, LZ}.

\smallskip

 Whether or not we can  use some  geometric symmetry property of $(NS)$ to study
 the general solutions of $(NS)$ seems to be an interesting question. Let us write the solution $u$ of $(NS)$
 in cylindrical coordinates as
\begin{equation}\label{ucylin}
u(t,x)=u^r(t,r,\th,z)\vv e_r+u^\theta(t,r,\th,z)\vv e_\theta+u^z(t,r,\th,z)\vv e_z.
\end{equation}
Different from the axially symmetric case, here $\ur,~\ut$ and $\uz$ may vary with respect to the
$\th$ variable. Notice that in cylindrical coordinates, one has
\beq\label{coordinate}
\na_x=\vv e_r\p_r+\vv e_\theta\f{\pa_\th}{r}+\vv e_z\p_z,
\quad \D_x=\p_r^2+\f{\p_r}{r}+\f{\p_\theta^2}{r^2}+\p_z^2,
\eeq
so that we can equivalently reformulate the system \eqref{NS} in cylindrical coordinates as
\begin{equation}\label{eqtu}
\left\{
\begin{aligned}
& D_t u^r-\bigl(\pa_r^2+\nu^2\pa_z^2+\f{\pa_r}{r}
+\f{\pa_\th^2}{r^2}-\f1{r^2}\bigr)u^r
-\f{\left(u^\theta\right)^2}{r}
+2\f{\pa_\th u^\theta}{r^2}+\pa_r P=0,\\
& D_t u^\theta-\bigl(\pa_r^2+\nu^2\pa_z^2+\f{\pa_r}{r}
+\f{\pa_\th^2}{r^2}-\f1{r^2}\bigr)u^\theta
+\f{u^r u^\theta}{r}-2\f{\pa_\th u^r}{r^2}
+\f{\pa_\th P}{r}=0,\\
& D_t u^z-\bigl(\pa_r^2+\nu^2\pa_z^2+\f{\pa_r}{r}
+\f{\pa_\th^2}{r^2}\bigr) u^z+\pa_z P=0,\\
& \pa_r u^r+\frac{u^r}{r}+\pa_z u^z
+\f{\pa_\th u^\theta}{r}=0,\quad (t,r,z)\in\R^+\times\wt\Omega,\\
&(u^r,u^\theta,u^z)|_{r=1}=0,
\quad(u^r,u^\theta,u^z)|_{t=0} =(u_{\rm in}^r,u^\theta_{\rm in},u^z_{\rm in}),
\end{aligned}
\right.
\end{equation}
where $D_t\eqdef\pa_t+ u\cdot\nabla_x=\pa_t+\bigl(u^r\pa_r+\f{u^\theta}{r}\pa_\th
+u^z\pa_z\bigr)$ denotes the material derivative and $\wt\Om\eqdefa  \bigl\{ (r,z)\in [0,1[\times\R \bigr\}.$

Concerning this question, lately the last two authors of this paper made some progresses  in \cite{LZ6}.
 In particular, the authors of \cite{LZ6}  investigate $(NS)$ in the whole space $\R^3$ and with initial data
\begin{equation}\label{1.5}
	\left\{
	\begin{split}
& \ur_{\rm in}(x)=a^r(r,z)\cos N\th+b^r(r,z)\sin N\th,\\
& \ut_{\rm in}(x)=N^{-1}\bigl(a^\th(r,z)\cos N\th
+b^\th(r,z)\sin N\th\bigr),\\
& \uz_{\rm in}(x)=a^z(r,z)\cos N\th+b^z(r,z)\sin N\th,
	\end{split}
	\right.
\end{equation}
where the profiles $\vv\al\eqdef(a^r,a^\th,a^z,b^r,b^\th,b^z)$ are axially symmetric functions satisfying
\begin{equation}\label{1.6}
\vv\al\in\bigl\{f(r,z):\
\bigl\|r^{\f12}f\bigr\|_{L^6(\R^3)}+\|f\|_{L^2(\R^3)}
+\bigl\|(\pa_rf,\pa_zf)\bigr\|_{L^2(\R^3)}
+\bigl\|r^{-1}f\bigr\|_{L^2(\R^3)}<\infty\bigr\},
\end{equation}
and the divergence-free condition for $u_{\rm in},$ which writes in terms of the $\vv\al,$ is
\beq\label{S1eq1}
\pa_r a^r+\f{a^r}{r}+\pa_z a^z
	+\f{b^\th}{r}=0,\quad\pa_r b^r
	+\f{b^r}{r}+\pa_z b^z-\f{a^\th}{r}=0.\eeq
Then there exists an integer $N_0,$ which depends only on the initial data, so that for $N\geq N_0,$
the system $(NS)$ has a unique global solution of the form:
\begin{equation}\label{solgeneral}
\left\{
\begin{split}
\ur(t,x)=&\ur_{0}(t,r,z)
+\sum_{k=1}^\infty\left(\ur_{k}(t,r,z)\cos kN\th
+\vr_{k}(t,r,z)\sin kN\th\right),\\
\ut(t,x)=&\ut_{0}(t,r,z)+\sum_{k=1}^\infty
\bigl(\ut_{k}(t,r,z)\cos kN\th
+\vt_{k}(t,r,z)\sin kN\th\bigr),\\
\uz(t,x)=&\uz_{0}(t,r,z)
+\sum_{k=1}^\infty\left(\uz_{k}(t,r,z)\cos kN\th
+\vz_{k}(t,r,z)\sin kN\th\right).
\end{split}
\right.
\end{equation}
Furthermore, for any positive time $t,$ this solution almost concentrates on the same frequency $N$ as the initial data \eqref{1.5}.
 Precisely, for any $p\in ]5,6[$ and $ 1<\beta_p<\f{p-1}4,~
0<\al_p<\f{p-3-2\beta_p}4$, the Fourier coefficients in \eqref{solgeneral} satisfy for $\varpi_k\eqdef\bigl(\ur_k,\vr_k,\uz_k,\vz_k, \sqrt{kN}\ut_k,\sqrt{kN}\vt_k\bigr)$ that
\begin{equation}\label{1.8}
\|u_0\|_{L^\infty(\R^+;L^3)}\lesssim N^{-\f15},
\andf\sum_{k=2}^\infty (kN)^{2\beta_p}
\bigl\|r^{1-\f3p}\varpi_{k}\bigr\|_{L^\infty(\R^+;L^p)}^p
\lesssim N^{-2\al_p}.
\end{equation}

The aim of this paper is first to extend the result in \cite{LZ6} to the cylindrical domain case,
the  result states as follows:

\begin{thm}\label{thm1}
{\sl Let $\delta\in [0,\f14[$ and $\vv\al\eqdef(a^r,a^\th,a^z,b^r,b^\th,b^z)\in H^{0,1}$
be axially symmetric functions satisfying \eqref{S1eq1}.
\footnote{In this paper, $f\in H^{0,m}$ means that both $f$ and $\pa_z^m f$ belong to $L^2$. And the norms in this paper are always taken on $\Omega$ for spacial variables without specially mentioning.}
Let the   initial data be given by
\begin{equation}\label{initialN}
	\left\{
	\begin{split}
& \ur_{\rm in}(x)=N^{\delta}\bigl(
a^r(r,z)\cos N\th+b^r(r,z)\sin N\th\bigr),\\
& \ut_{\rm in}(x)=N^{\delta-1}\bigl(
a^\th(r,z)\cos N\th+b^\th(r,z)\sin N\th\bigr),\\
& \uz_{\rm in}(x)=N^{\delta}\bigl(
a^z(r,z)\cos N\th+b^z(r,z)\sin N\th\bigr).
	\end{split}
	\right.
\end{equation}
Then
for any $\eta\in[0,\f12-\delta[$\,, if $N\in\N^+$ is so large that
\begin{equation}\label{smallness}
\bigl(N^{-(\f14-\delta)}+
N^{-(\f12-\delta-\eta)}\bigr)
\|\vv\al\|_{L^2}^{\f12}
\|\p_z\vv\al\|_{L^2}^{\f12}\leq \e
\end{equation}
 for some $\e>0$ being sufficiently   small, the system $(NS)$ has a unique global strong solution $u\in C([0,\infty[; H^{0,1})$ with $\nabla u\in L^2(\R^+;H^{0,1})$, which is of the form  \eqref{solgeneral} and with the Fourier
 coefficients of which, $\vv u_0=(\ur_0,\ut_0,\uz_0)$ and $\vv u_k=\bigl(\ur_k,\vr_k,\uz_k,\vz_k,\ut_k, \vt_k\bigr),$ satisfy
\begin{equation}\begin{split}\label{S1eqw}
\|\p_z^j \vv u_0(t)\|_{L^2}
\leq C N^{-\f14+\delta}
\|\p_z^j\vv\al\|_{L^2},\quad
\|\p_z^j\vv u_k(t)\|_{L^2}
\leq C k^{-1}N^{-\eta(k-1)+\delta}
\|\pa_z^j\vv\al\|_{L^2},
\end{split}\end{equation} for any $k\in\N^+,  t>0$ and $j\in\{0,1\}.$
}\end{thm}

We emphasize that the proof of Theorem  \ref{thm1} works for $(NS)$ in the whole space case,
which is different from Theorem \ref{thm2} for $(ANS)$ (see Remark \ref{S1rmk3} below).

\begin{rmk}
{\sl Compared with the previous result in \cite{LZ6},
 Theorem \ref{thm1} has the following improvements:
 \begin{enumerate}

\item[(1)] It is well-known that the decay property of the Fourier coefficients to a function is related to its regularity.
 In \eqref{S1eqw}, the $kN$-th Fourier coefficient $\vv u_k$ decays exponentially as $N^{-\eta(k-1)+\delta}$, whereas in \eqref{1.8},
  the Fourier coefficients  only have polynomial decay in $k$.

\item[(2)] Here the profile of the initial data, $\vv\al$, belongs to $H^{0,1}$, which is much weaker than the regularity assumptions in \cite{LZ6} listed in \eqref{1.6}.

\item[(3)] Comparing \eqref{initialN} with \eqref{1.5}, it is easy to observe that the size of the initial data in Theorem \ref{thm1} is $N^\delta$ times that in \cite{LZ6}. We are wondering
    \end{enumerate}

    \smallskip
\noindent{\bf Open problem:} How large can this $\delta$ be? Can we improve Theorem \ref{thm1} to the case when $\delta=1$?
\smallskip

 This problem seems very  challenging and interesting. Indeed
 as we mentioned before,  so far we can only prove the global well-posedness of $(NS)$ with axially symmetric data
  with swirl part of which is sufficiently small. If one can improve
    Theorem \ref{thm1} to the case for $\delta=1,$ then one has example of initial data  which generates a unique global soution
    of $(NS)$ and which has large swirl part.
}\end{rmk}

\smallskip

Next let us consider  anisotropic Navier-Stokes equations $(ANS).$
Due to the anisotropic property of the system $(ANS),$ it is natural to work the well-posedness of
the system  with initial data in
anisotropic type functional spaces.
In fact, Chemin et al. \cite{CDGG} first
 proved that the system $(ANS)$ has a local solution
with initial data in ${H}^{0,\f12+\varepsilon}$ for some $\varepsilon>0$,
and the solution can be extended globally in time  if in addition
\begin{equation}\label{smallCDGG}
\|u_0\|_{L^2}^\varepsilon
\|u_0\|_{\dot{H}^{0,\f12+\varepsilon}}^{1-\varepsilon}\leq c \with \|u_0\|_{\dot{H}^{0,s}}
\eqdefa \||\xi_3|^{s}\widehat{u_0}(\xi)\|_{L^2},
\end{equation}
for some sufficiently small constant $c$. The uniqueness part was proved in \cite{Iftimie}.
This well-posedness result of $(ANS)$ in \cite{CDGG} was improved to the critical spaces in \cite{Pa02} for positive index case
and \cite{CZ07} for the negative index case. One may check \cite{LPZ,PZ1,Zhang10} for further improvements.

Our second result concerns the global well-posedness of the system $(ANS)$ with initial data of the form \eqref{initialN}. Precisely

\begin{thm}\label{thm2}
{\sl Let $\delta\in\bigl[0,\f14\bigr[$ and $m\geq3$ be an integer. Let $\vv{\al}\in H^{0,m+1}$ be axially symmetric functions satisfying \eqref{S1eq1}.
 Then there exists some small generic  positive constant $\e$ such that for any $\eta\in[\f14,\f12-\delta[$\,, if $N\in\N^+$ is so large that
\begin{equation}\label{smallANScyl}
\bigl(N^{-\left(\f14-\delta\right)}
+N^{-\f1m(1-\eta)}\bigr)\sum_{j=0}^{m+1}
\|\pa_z^j\vv\al\|_{L^2}\leq\e,\quad
N^{-\left(\f12-\delta-\eta\right)}
\|\vv\al\|_{L^2}^{\f12}
\|\p_z\vv\al\|_{L^2}^{\f12}\leq \e,
\end{equation}
 the system $(ANS)$ with initial data \eqref{initialN} has a unique global strong solution $u$ of the  form \eqref{solgeneral},
 and $u\in C([0,\infty[; H^{0,m})$ with $\nablah u\in L^2(\R^+; H^{0,m})$, which satisfies
\begin{equation}\label{HmANScyl}
\|u\|_{L^\infty(\R^+;H^{0,m})}^2
+\|\nablah u\|_{L^2(\R^+;H^{0,m})}^2
\leq N^{2\delta}\|\vv\al\|_{H^{0,m}}^2
+N^{-4(\f14-\delta)}
\|\vv\al\|_{H^{0,m+1}}^4.
\end{equation}
 Moreover, for any $\sigma\in\bigl]\f1{2m-3},\f12\bigr[$\,, there exists some constant $C_\sigma$ depending only on $\sigma$, so that the Fourier coefficients of $u$ satisfy for $j\in\{0,1\}$ and $t>0$:
\begin{equation}\label{asymthm1.2}\begin{split}
&\|\p_z^j \vv u_0(t)\|_{L^2}
\leq C_\sigma N^{-\f14+\delta}
\|\p_z^j\vv\al\|_{L^2},\\
&\|\p_z^j\vv u_k(t)\|_{L^2}
\leq C_\sigma k^{-\sigma(m-j)}N^{-\eta(k-1)+\delta}
\|\pa_z^j\vv\al\|_{L^2}\quad\text{for}\quad
1\leq k\leq A_j,\\
&\|\p_z^j\vv u_k(t)\|_{L^2}
\leq C_\sigma k^{-\sigma(m-j)}
N^{-\left(\f12-\eta-\delta\right)(m-j)-\eta+\delta}
\|\pa_z^j\vv\al\|_{L^2}\quad\text{for}\quad
k>A_j,
\end{split}\end{equation}
where $A_j\eqdefa\bigl[\f{(\f12-\eta-\delta)(m-j)}\eta\bigr]+2$.
}\end{thm}

\begin{rmk} \begin{itemize}
\item[(1)] Let us compare the decay-in-$k$ estimates \eqref{asymthm1.2} with \eqref{S1eqw}.

For $k\leq A_j$, $\vv u_k$ in \eqref{asymthm1.2} decays exponentially as $N^{-\eta(k-1)+\delta}$, which  coincides with \eqref{S1eqw}. Noticing that $A_j$ grows to infinity as $m\rightarrow\infty$, which indicates that the more regular $u_{\rm in}$  in $x_3$ variable, the faster $\vv u_k$
  decays with respect to $k$. This phenomenon also coincides with the fact that there is no smoothing effect in $x_3$ variable for the anisotropic Navier-Stokes equations.

On the other hand, in \eqref{S1eqw}, due to the exponential decay $N^{-\eta(k-1)+\delta}$, the power of $k$ in the polynomial part is actually not important, namely, there makes no difference whether it is $k^{-1}$ or $k^{-100}$ in \eqref{S1eqw}. However,  the exponential decay for $k$ in \eqref{asymthm1.2} is in fact truncated to be a constant for all $k> A_j$,  the decay power in the polynomial part is indeed crucial. Therefore,
for the anisotropic case, we need to construct the energy functionals more subtly (see \eqref{defDj} below).

\item[(2)] By interpolating between \eqref{HmANScyl} and \eqref{asymthm1.2}, we obtain the following polynomial decay-in-$k$ estimates for
$\vv u_k(t)$:
$$\|\p_z^j\vv u_k(t)\|_{L^2}\leq \|\vv u_k(t)\|_{L^2}^{1-\f{j}m}\|\p_z^m\vv u_k(t)\|_{L^2}^{\f{j}m}
\leq C_0\, k^{-\sigma(m-j)},\quad  \forall \ t>0,$$ for any $ j\in\{0,1,\cdots,m\},~ k\in\N^+,$
where $C_0$ is some positive constant depending only on $\sigma,~\eta,~\delta,~N$ and $\|\vv\al\|_{H^{0,m+1}}$.
\end{itemize}
\end{rmk}

\begin{rmk}\label{S1rmk3}
Different from Theorem \ref{thm1}, the proof of Theorem \ref{thm2} heavily relies on the fact that for any function $f$ defined in $\Omega$, there holds the following point-wise inequality
\begin{equation}\label{pointwisecyl}
|f(x)|\leq\bigl|\f{f(x)}r\bigr|,
\quad\forall\ x\in \Omega,
\end{equation}
which obviously does not hold for the whole space case. Thus to establish similar results as Theorem \ref{thm2} for $(ANS)$ in the whole space  requires a different approach.
\end{rmk}

Let us end this section with some notations that will be used throughout this paper.

\noindent{\bf Notations:}  $(\cdot | \cdot)$ designates the $L^2$ inner product on $\Omega.$
We always denote $C$ to be a uniform constant
which may vary from line to line.
$a\lesssim b$ means that $a\leq Cb$.
We use subscript $\h$ (resp. $\v$) to denote  the norm being taken
on $B_\h\eqdefa \bigl\{ (x_1,x_2)\in\R^2,\ x_1^2+x_2^2\leq 1\ \bigr\}$ (resp. $\R_{x_3}$), for example
$L^p_\v(L^q_\h)$ means $L^p(\R_{x_3}; L^q(B_\h))$.
Finally
all the $L^p$ norm stands for the $L^p(\Omega)$ norm so that
$\|f\|_{L^p}=\Bigl(\int_0^1\int_0^{2\pi}\int_{-\infty}^\infty |f(r,\theta,z)|^p \,rdr\,d\theta\,dz\Bigr)^{\f1p}.$

\setcounter{equation}{0}
\section{Ideas of the proof}

In this section, we shall first derive the equations for the Fourier coefficients $\vv u_0$ and $\vv u_k$ in \eqref{solgeneral}.
And then we shall sketch the proofs of Theorems \ref{thm1} and \ref{thm2}.

\subsection{Derivation of the equations for $\vv u_0$ and $\vv u_k$}
In view of the initial data \eqref{initialN},
we  deduce from Theorem 2.1 in \cite{LZ6} that the solutions of the system \eqref{eqtu}  can be expanded into Fourier series of the form \eqref{solgeneral}, correspondingly we write  the pressure function $P$ as
\begin{equation}\label{Pexpan}
P(t,x)=P_{0}(t,r,z)
+\sum_{k=1}^\infty \left(P_{k}(t,r,z)\cos kN\th+ Q_{k}(t,r,z)\sin kN\th\right).
\end{equation}
One should keep in mind that we always use the subscript $k$ to denote the $kN$-th Fourier coefficient for notational simplification.
And in what follows, we always denote
\beq \label{S2notion}
\begin{split}
u_0\eqdef& \ur_0\vv e_r+\uz_0\vv e_z,
\quad\wt u_k\eqdef\ur_k\vv e_r+\uz_k\vv e_z,
\quad\wt v_k\eqdef\vr_k\vv e_r+\vz_k\vv e_z,\\
&\wt\nabla\eqdef\vv e_r\pa_r+\vv e_z\pa_z \andf D_{0,t}\eqdef \p_t+u_0^r\pa_r+u_0^z\pa_z.
\end{split}
\eeq

By substituting \eqref{solgeneral} and \eqref{Pexpan} into the system \eqref{eqtu} and using the identities
\begin{align*}
&2\cos k_1\th\cdot\cos k_2\th=\cos (k_1+k_2)\th+\cos (k_1-k_2)\th,\\
&2\sin k_1\th\cdot\sin k_2\th=\cos (k_1-k_2)\th-\cos (k_1+k_2)\th,\\
&2\sin k_1\th\cdot\cos k_2\th=\sin (k_1+k_2)\th+\sin (k_1-k_2)\th,
\end{align*}
and then comparing the Fourier coefficients of the resulting equations, we find
\begin{itemize}
\item
$(\ur_0,~\ut_0 ,~\uz_0, ~P_0)$ satisfies
\end{itemize}
\begin{equation}\label{eqtu0}
\left\{
\begin{split}
&D_{0,t}\ur_{0}
-\bigl(\p_r^2+\f{\p_r}r-\f{1}{r^2}+\nu^2\p_z^2
\bigr)\ur_{0}+\pa_r P_{0}\\
&\qquad
=\f{(\ut_0)^2}r-\f12\sum_{k=1}^\infty\Bigl(\wt u_k\cdot\wt\nabla\ur_{k}+\wt v_k\cdot\wt\nabla\vr_{k}
+\f{kN}r (\ut_k \vr_k-\vt_k\ur_k)-\f{(\ut_k)^2+{(\vt_{k})^2}}r\Bigr),\\
&D_{0,t}\ut_{0}
-\bigl(\p_r^2+\f{\p_r}r-\frac{1}{r^2}+\nu^2\p_z^2
\bigr)\ut_{0}\\
&\qquad=-\f{\ut_0\ur_0}r-\f12\sum_{k=1}^\infty\Bigl(\wt u_k\cdot\wt\nabla\ut_{k}+\wt v_k\cdot\wt\nabla\vt_{k}
+\f1r(\ur_k\ut_k+\vr_k\vt_k)\Bigr),\\
&D_{0,t}\uz_{0}
-\bigl(\p_r^2+\f{\p_r}r+\nu^2\p_z^2\bigr)
\ur_{0}+\pa_z P_{0}\\
&\qquad=-\f12\sum_{k=1}^\infty\Bigl(\wt u_k\cdot\wt\nabla\uz_{k}+\wt v_k\cdot\wt\nabla\vz_{k}
+\f{kN}r (\ut_k \vz_k-\vt_k\uz_k)\Bigr),\\
& \pa_r\ur_{0}+\frac{\ur_{0}}r+\pa_z\uz_{0}=0, \quad (t,r,z)\in\R^+\times\wt\Omega,
\\
&\bigl(\ur_0,\ut_0,\uz_{0}\bigr)|_{r=1}=0,
\quad \bigl(\ur_{0},\ut_0,\uz_{0}\bigr)|_{t=0}=0.
\end{split}
\right.
\end{equation}

\begin{itemize}
\item
 $\left(\ur_{k},~\vt_{k},~\uz_{k},~P_k\right)$ for $k\in\N^+$ satisfies
 \end{itemize}
\begin{equation}\label{eqtuk}
\left\{
\begin{split}
&D_{0,t}\ur_{k}
-\bigl(\p_r^2+\f{\p_r}r-\frac{1+k^2N^2}{r^2}
+\nu^2\p_z^2\bigr)\ur_{k}
+\f{2kN}{r^2}\vt_{k}+\pa_r P_{k} \\
&\qquad\qquad=-\wt u_k\cdot \wt \nabla  \ur_0 -\f{kN}r \ut_0\vr_k +\f2r \ut_0 \ut_k+ F^r_{k},\\
&D_{0,t}\vt_{k}
-\bigl(\p_r^2+\f{\p_r}r-\frac{1+k^2N^2}{r^2}
+\nu^2\p_z^2\bigr)\vt_{k}
+\f{2kN}{r^2}\ur_{k}-\f{kN}r P_{k} \\
&\qquad\qquad=-\wt v_k\cdot \wt \nabla  \ut_0 +\f{kN}r \ut_0\ut_k -\f1r (\ur_0 \vt_k+\ut_0\vr_k)+ F^\th_{k},\\
&D_{0,t}\uz_{k}
-\bigl(\p_r^2+\f{\p_r}r-\f{k^2N^2}{r^2}+\nu^2\p_z^2\bigr)\uz_{k}
+\pa_z P_{k} =-\wt u_k\cdot \wt \nabla  \uz_0 -\f{kN}r \ut_0\vz_k + F^z_{k},\\
& \pa_r\ur_{k}+\frac{\ur_k}r+\pa_z\uz_{k}
+kN\f{\vt_{k}}r=0,\quad (t,r,z)\in\R^+\times\wt\Omega,\\
&(\ur_{k},\vt_{k},\uz_{k})|_{r=1}=0,\\
& (\ur_{1},\vt_{1},\uz_{1})|_{t=0}=N^\delta\bigl(a^r, N^{-1}{b^\th},a^z\bigr)
\andf(\ur_{k},\vt_{k},\uz_{k})|_{t=0}=0
~\text{ for }~k\geq2,
\end{split}
\right.
\end{equation}
where the external force terms $F^r_{k},~F^\th_{k},~F^z_{k},$ are given respectively by
\beq\label{eqtukq}
\begin{split}
F^r_k\eqdefa&-\f12\sum_{k_1+k_2=k} \Bigl(
\wt u_{k_1}\cdot\wt\nabla\ur_{k_2}-\wt v_{k_1}\cdot\wt\nabla\vr_{k_2}
+\f{k_2 N}r(\ut_{k_1}\vr_{k_2}+\vt_{k_1}\ur_{k_2})\\
&\quad
-\f{1}r(\ut_{k_1}\ut_{k_2}-\vt_{k_1}\vt_{k_2})\Bigr)-\f12\sum_{|k_1-k_2|=k} \Bigl(
\wt u_{k_1}\cdot\wt\nabla\ur_{k_2}+\wt v_{k_1}\cdot\wt\nabla\vr_{k_2}\\
&\quad+\f{k_2 N}r(\ut_{k_1}\vr_{k_2}-\vt_{k_1}\ur_{k_2})
-\f{1}r(\ut_{k_1}\ut_{k_2}+\vt_{k_1}\vt_{k_2})\Bigr),\\
F^\th_k\eqdefa&-\f12\sum_{k_1+k_2=k} \Bigl(
\wt u_{k_1}\cdot\wt\nabla\vt_{k_2}+\wt v_{k_1}\cdot\wt\nabla\ut_{k_2}
+\f{k_2 N}r(\vt_{k_1}\vt_{k_2}-\ut_{k_1}\ut_{k_2})\\
&\quad+\f{1}r(\ur_{k_1}\vt_{k_2}+\vr_{k_1}\ut_{k_2})\Bigr)+\f12\Bigl(\sum_{k_1-k_2=k}-\sum_{k_2-k_1=k}\Bigr) \Bigl(
\wt u_{k_1}\cdot\wt\nabla\vt_{k_2}\\
&\quad-\wt v_{k_1}\cdot\wt\nabla\ut_{k_2}-\f{k_2 N}r(\ut_{k_1}\ut_{k_2}+\vt_{k_1}\vt_{k_2})+\f{1}r(\ur_{k_1}\vt_{k_2}-\vr_{k_1}\ut_{k_2})\Bigr),\\
F^z_k\eqdefa&-\f12\sum_{k_1+k_2=k} \Bigl(
\wt u_{k_1}\cdot\wt\nabla\uz_{k_2}-\wt v_{k_1}\cdot\wt\nabla\vz_{k_2}
+\f{k_2 N}r(\ut_{k_1}\vz_{k_2}+\vt_{k_1}\uz_{k_2})\Bigr)\\
&-\f12\sum_{|k_1-k_2|=k} \Bigl(
\wt u_{k_1}\cdot\wt\nabla\uz_{k_2}+\wt v_{k_1}\cdot\wt\nabla\vz_{k_2}
+\f{k_2 N}r(\ut_{k_1}\vz_{k_2}-\vt_{k_1}\uz_{k_2})\Bigr),\\
\end{split} \eeq
and $k_1,k_2\in\N^+$ in the above summations.

\begin{itemize}
\item
 $\left(\vr_{k},~\ut_{k},~\vz_{k},~Q_k\right)$ for $k\in\N^+$ satisfies
 \end{itemize}
\begin{equation}\label{eqtvk}
\left\{
\begin{split}
&D_{0,t}\vr_{k}
-\bigl(\p_r^2+\f{\p_r}r-\frac{1+k^2N^2}{r^2}
+\nu^2\p_z^2\bigr)\vr_{k}
-\f{2kN}{r^2}\ut_{k}+\pa_r Q_{k} \\
&\qquad\qquad=-\wt v_k\cdot \wt \nabla  \ur_0 +\f{kN}r \ut_0\ur_k +\f2r \ut_0 \vt_k+ G^r_{k},\\
&D_{0,t}\ut_{k}
-\bigl(\p_r^2+\f{\p_r}r-\frac{1+k^2N^2}{r^2}
+\nu^2\p_z^2\bigr)\ut_{k}
-\f{2kN}{r^2}\vr_{k}+\f{kN}r Q_{k} \\
&\qquad\qquad=-\wt u_k\cdot \wt \nabla  \ut_0 -\f{kN}r \ut_0\vt_k -\f1r (\ur_0 \ut_k+\ut_0\ur_k)+ G^\th_{k},\\
&D_{0,t}\vz_{k}
-\bigl(\p_r^2+\f{\p_r}r-\frac{k^2N^2}{r^2}+\nu^2\p_z^2\bigr)\vz_{k}
+\pa_z Q_{k} =-\wt v_k\cdot \wt \nabla  \uz_0 +\f{kN}r \ut_0\uz_k + G^z_{k},\\
& \pa_r\vr_{k}+\frac{\vr_k}r+\pa_z\vz_{k}
-kN\f{\ut_{k}}r=0,\quad (t,r,z)\in\R^+\times\wt\Omega,\\
&(\vr_{k},\ut_{k},\vz_{k})|_{r=1}=0,\\
&(\vr_{1},\ut_{1},\vz_{1})|_{t=0}=N^\delta\bigl(b^r,N^{-1}{a^\th},b^z\bigr)
\andf(\vr_{k},\ut_{k},\vz_{k})|_{t=0}=0
~\text{ for }~k\geq2,
\end{split}
\right.
\end{equation}
where the external force terms $G^r_{k},~G^\th_{k},~G^z_{k},$ are given respectively by
\beq\label{eqtvkq}
\begin{split}
G^r_k\eqdefa&-\f12\sum_{k_1+k_2=k} \Bigl(
\wt u_{k_1}\cdot\wt\nabla\vr_{k_2}+\wt v_{k_1}\cdot\wt\nabla\ur_{k_2}
+\f{k_2 N}r(\vt_{k_1}\vr_{k_2}-\ut_{k_1}\ur_{k_2})\\
&\quad-\f{1}r(\ut_{k_1}\vt_{k_2}+\vt_{k_1}\ut_{k_2})\Bigr)+\f12\Bigl(\sum_{k_1-k_2=k}-\sum_{k_2-k_1=k}\Bigr) \Bigl(
\wt u_{k_1}\cdot\wt\nabla\vr_{k_2}\\
&\quad-\wt v_{k_1}\cdot\wt\nabla\ur_{k_2}
-\f{k_2 N}r(\ut_{k_1}\ur_{k_2}+\vt_{k_1}\vr_{k_2})
-\f{1}r(\ut_{k_1}\vt_{k_2}-\vt_{k_1}\ut_{k_2})\Bigr),\\
G^\th_k\eqdefa&-\f12\sum_{k_1+k_2=k} \Bigl(
\wt u_{k_1}\cdot\wt\nabla\ut_{k_2}-\wt v_{k_1}\cdot\wt\nabla\vt_{k_2}
+\f{k_2 N}r(\ut_{k_1}\vt_{k_2}+\vt_{k_1}\ut_{k_2})\\
&\quad+\f{1}r(\ur_{k_1}\ut_{k_2}-\vr_{k_1}\vt_{k_2})\Bigr)-\f12\sum_{|k_1-k_2|=k} \Bigl(
\wt u_{k_1}\cdot\wt\nabla\ut_{k_2}+\wt v_{k_1}\cdot\wt\nabla\vt_{k_2}\\
&\quad
+\f{k_2 N}r(\ut_{k_1}\vt_{k_2}-\vt_{k_1}\ut_{k_2})
+\f{1}r(\ur_{k_1}\ut_{k_2}+\vr_{k_1}\vt_{k_2})\Bigr),\\
G^z_k\eqdefa&-\f12\sum_{k_1+k_2=k} \Bigl(
\wt u_{k_1}\cdot\wt\nabla\vz_{k_2}+\wt v_{k_1}\cdot\wt\nabla\uz_{k_2}
+\f{k_2 N}r(\vt_{k_1}\vz_{k_2}-\ut_{k_1}\uz_{k_2})\Bigr)\\
&+\f12\Bigl(\sum_{k_1-k_2=k}-\sum_{k_2-k_1=k}\Bigr) \Bigl(
\wt u_{k_1}\cdot\wt\nabla\vz_{k_2}-\wt v_{k_1}\cdot\wt\nabla\uz_{k_2}
-\f{k_2 N}r(\ut_{k_1}\uz_{k_2}+\vt_{k_1}\vz_{k_2})\Bigr),\\
\end{split} \eeq
and $k_1,k_2\in\N^+$ in the above summations.

\subsection{Outline of the proof to Theorem \ref{thm1}}

Recall that $\vv u_0=(\ur_0,\ut_0,\uz_0),~\vv u_k= (\ur_k,\vt_k,\uz_k,$ $\vr_k,\ut_k,\vz_k),$
$~0\leq\delta<\f14$ and $0\leq\eta<\f12-\delta$ in the statement of Theorem \ref{thm1}. For $j\in\{0,1\},$
we introduce the following weighted  energy functional for the Fourier coefficients:
\begin{equation}\begin{aligned}\label{def:E}
E_j(t)\eqdefa &N^{2(\f14-\eta)}
\Bigl(\|\pa_z^j \vv u_0\|_{L^\oo_t(L^2)}^2
+\|\wt\nabla\pa_z^j \vv u_0\|_{L^2_t(L^2)}^2
+\bigl\|\f{(\pa_z^j\ur_0,\pa_z^j \ut_0)}r\bigr\|_{L^2_t(L^2)}^2\Bigr)\\
&+\sup_{k\geq1}k^2
N^{2\eta (k-2)} \Bigl(\|\pa_z^j \vv u_k\|_{L^\oo_t(L^2)}^2+\|\wt\nabla\pa_z^j \vv u_k\|_{L^2_t(L^2)}^2+\f{k^2 N^2}2 \bigl\|\f{\pa_z^j \vv u_k}r\bigr\|_{L^2_t(L^2)}^2\Bigr).
\end{aligned}\end{equation}

The following {\it a priori} estimates are the cores of the proof of Theorem \ref{thm1}.

\begin{prop}\label{main prop}
{\sl
Let  $\nu=1$ and $N\geq2$, $\vv u_0$ and $\vv u_k$ be smooth enough solutions of \eqref{eqtu0}, \eqref{eqtuk} and \eqref{eqtvk} respectively on $[0,T]$. Then
 for any $t\leq T$ and $j\in\{0,1\},$ there holds
\begin{equation}\begin{split}\label{eq:E}
& E_j(t) \leq E_j(0)+ C\max\bigl\{N^{\eta-\f14},N^{2(\eta-\f14)} \bigr\} E_0^\f14(t)E_1^\f14(t)E_j(t).
\end{split}\end{equation}
}
\end{prop}

The proof of Proposition \ref{main prop} will be presented in Section \ref{sec4}.

 With Proposition \ref{main prop}, Theorem \ref{thm1} follows from a standard continuity argument.

\begin{proof}[Proof of Theorem \ref{thm1}] We first get, by using standard energy method that
the systems \eqref{eqtu0}, \eqref{eqtuk} and \eqref{eqtvk} for $\nu=1$ have unique local solutions, which can be
in fact be justified by the following {\it a priori} estimates. For simplicity, here we just present the
{\it a priori} estimates. We define
\begin{equation}\label{2.10}
T^\star_1\eqdef \sup\bigl\{\,t>0:\, E_0(t)\leq 4E_0(0)\andf E_1(t) \leq 4E_1(0)\, \bigr\}.
\end{equation}
 We are going to prove that $T^\star_1=\infty$ under the assumption \eqref{smallness}. Indeed
it follows from \eqref{smallness}, \eqref{def:E} and \eqref{2.10} that for any $t\leq T^\star_1$,
\begin{align*}
\max\bigl\{N^{\eta-\f14},N^{2(\eta-\f14)} \bigr\}
E_0^{\f14}(t)E_1^{\f14}(t)
&\leq 2\max\bigl\{N^{\eta-\f14},N^{2(\eta-\f14)} \bigr\}
 E_0^{\f14}(0)E_1^{\f14}(0)\\
&\leq 2\max\bigl\{N^{\delta-\f14},
N^{\eta+\delta-\f12} \bigr\}\|\vv\al\|_{L^2}^{\f12}
\|\p_z\vv\al\|_{L^2}^{\f12}
\leq 2\e.
\end{align*}
Then we deduce from \eqref{eq:E} that for any $t\leq T^\star_1$ and $j\in\{0,1\}$
$$E_j(t) \leq E_j(0)+ 2C\e E_j(t).$$
In particular, if $\e$ is so small  that
$2C\e\leq\f12$, we infer
\begin{equation}\label{2.11}
E_j(t)\leq 2 E_j(0),\quad\mbox{for}\ \ j\in\{0,1\}\andf t\leq T_1^\star,
\end{equation}
which contradicts with \eqref{2.10} unless $T^\star_1=\infty.$ This in turn shows that $T^\star_1=\infty$ if
$\e$ is sufficiently small in \eqref{smallness}, and there holds \eqref{2.11}  for any $t>0$. This completes the proof Theorem \ref{thm1}.
\end{proof}

\subsection{Outline of the proof to Theorem \ref{thm2}}
The proof of Theorem \ref{thm2}  is more subtle due to the loss of dissipation in the vertical variable for the system $(ANS).$
 Toward this, we shall first derive the following rough estimates in Section \ref{sec5}:

\begin{prop}\label{propANScyl}
{\sl Under the assumptions of Theorem \ref{thm2}, the system  $(ANS)$ has a unique global solution $u\in C([0,\infty[; H^{0,m})$ with $\nablah u\in L^2(\R^+; H^{0,m}),$ which is of the form \eqref{solgeneral} and which satisfies for $0\leq \ell\leq m$
\begin{align*}
&\|\pa_z^\ell\vv u_0\|_{L^\infty(\R^+; L^2)}^2
+\|\pa_r\pa_z^\ell\vv u_0\|_{L^2(\R^+; L^2)}^2
+\|r^{-1}\pa_z^\ell
(u_0^r,\ut_0)\|_{L^2(\R^+; L^2)}^2
\lesssim \cH_\ell(\vv\al),
\\
&\|\p_z^\ell \vv u_1\|_{L^\infty(\R^+;L^2)}^2
+\|\pa_r\p_z^\ell\vv u_1\|_{L^2(\R^+;L^2)}^2
+N^2\|r^{-1}\p_z^\ell \vv u_1\|_{L^2(\R^+;L^2)}^2
\lesssim N^{2\delta}\|\p_z^\ell \vv\al\|_{L^2}^2
+\cH_\ell(\vv\al),
\end{align*}
and for  $k\geq2$
\begin{align*}
\|\pa_z^\ell\vv u_k\|_{L^\infty(\R^+; L^2)}^2
+\|\pa_r\pa_z^\ell\vv u_k\|_{L^2(\R^+; L^2)}^2
+k^2N^2\|r^{-1}\pa_z^\ell\vv u_k\|_{L^2(\R^+; L^2)}^2
\lesssim \cH_\ell(\vv\al),
\end{align*}
where
$$\cH_\ell(\vv\al)\eqdef N^{-4(\f14-\delta)}
\sum_{i=0}^{\max\{\ell,1\}}\|\pa_z^i\vv\al\|_{L^2}^4
+N^{-4(1-\delta)}\sum_{j=1}^{\ell+1}
\|\pa_z^j\vv\al\|_{L^2}^4.$$
}\end{prop}

Proposition \ref{propANScyl} yields the global existence part of Theorem \ref{thm2}.
It remains to  prove the decay-in-$k$ estimate \eqref{asymthm1.2}. In order to do so, for $\vv u_0$ and $\vv u_k$ being
determined by the systems \eqref{eqtu0}, \eqref{eqtuk} and \eqref{eqtvk} for $\nu=0,$ we
introduce for $j=0$ or $1$ the following energy functionals:
\begin{equation}\begin{split}\label{defDj}
D_j(t)\eqdefa&N^{2\left(\f14-\eta\right)}
\Bigl(\|\p_z^j \vv u_0\|_{L^\oo_t(L^2)}^2
+\|\pa_r\p_z^j \vv u_0\|_{L^2_t(L^2)}^2
+\bigl\|\f{(\p_z^j\ur_0,\p_z^j \ut_0)}r\bigr\|_{L^2_t(L^2)}^2\Bigr)\\
&+\sup_{k\in\N^+}\Bigl\{k^{2\sigma(m-j)}
N^{2\min\left\{\eta(k-2),\left(\f12-\eta-\delta\right)(m-j)\right\}}\\
&\qquad\qquad\times\Bigl(\|\p_z^j \vv u_k\|_{L^\oo_t(L^2)}^2+\|\pa_r\p_z^j \vv u_k\|_{L^2_t(L^2)}^2+\f{k^2 N^2}2\bigl\|\f{\p_z^j \vv u_k}r\bigr\|_{L^2_t(L^2)}^2\Bigr)\Bigr\}.
\end{split}\end{equation}

Concerning the energy functionals $D_j$ for $j=1,2,$ we shall prove in  Section \ref{sec6} the following {\it a priori} estimates:

\begin{prop}\label{propANScylE}
{\sl Let $u$ be  the unique solution of $(ANS)$ obtained in Proposition \ref{propANScyl} and $D_j, j=1, 2,$ be the energy functional
  defined by \eqref{defDj}. Then for any $t>0,$ there holds
\begin{equation}\begin{split}\label{ANScylE1}
D_0(t)\leq D_0(0)+C\Bigl(&N^{\max\left\{\eta-\f14,2\left(\eta-\f14\right)\right\}}
D_0^{\f1{2m}}(t)+N^{\f{m-1}{m}\left(\eta-\f14\right)}
\bigl\|\f{\pa_z^m (\ur_0,\ut_0)}r\bigr\|_{L^2_t(L^2)}^{\f1m}\\
&+N^{\eta-\delta}\sup_{k\in\N^+}
\bigl\|\f{\pa_z^m\vv u_k}r\bigr\|
_{L^2_t(L^2)}^{\f1m}\Bigr)
D_0^{\f54-\f1{2m}}(t)D_1^\f14(t),
\end{split}\end{equation}
and
\begin{equation}\begin{split}\label{ANScylE2}
D_1(t) \leq D_1(0)
+C\Bigl(&N^{\max\left\{\eta-\f14,2\bigl(\eta-\f14\bigr)\right\}}
D_1^{\f1{2(m-1)}}(t)
+N^{\f{m-2}{m-1}\left(\eta-\f14\right)}
\bigl\|\f{\pa_z^m (\ur_0,\ut_0)}r\bigr\|_{L^2_t(L^2)}^{\f1{m-1}}\\
&+N^{\eta-\delta}\sup_{k\in\N^+}
\bigl\|\f{\pa_z^m\vv u_k}r\bigr\|
_{L^2_t(L^2)}^{\f1{m-1}}\Bigr)
D_0^\f14(t) D_1^{\f54-\f1{2(m-1)}}(t).
\end{split}\end{equation}
}\end{prop}

 We are now in a position to complete the proof of Theorem \ref{thm2}.

\begin{proof}[Proof of Theorem \ref{thm2}] Once again we shall only present the {\it a priori} estimates
for smooth enough solutions of the systems \eqref{eqtu0}, \eqref{eqtuk} and \eqref{eqtvk} for $\nu=0.$
Let us denote
\begin{equation}\label{2.15}
T^\star_2\eqdef \sup\bigl\{\,t>0:\, D_0(t)\leq 4D_0(0)\andf D_1(t) \leq 4D_1(0)\, \bigr\}.
\end{equation}
Then we get, by using Proposition \ref{propANScyl} and the assumption \eqref{smallANScyl},  that for any $t\leq T^\star_2$,
\begin{align*}
N^{\max\left\{\eta-\f14,2\left(\eta-\f14\right)\right\}}
D_0^\f14(t) D_1^{\f14}(t)
&\leq 2N^{\max\left\{\delta-\f14,\eta+\delta-\f12\right\}}
\|\vv\al\|_{L^2}^{\f12}
\|\p_z\vv\al\|_{L^2}^{\f12}\\
&\leq 2\e,
\end{align*}
and
\begin{align*}
&\Bigl(N^{\f{m-1}{m}\left(\eta-\f14\right)}
\bigl\|\f{\pa_z^m (\ur_0,\ut_0)}r\bigr\|_{L^2_t(L^2)}^{\f1m}
+N^{\eta-\delta}\sup_{k\in\N^+}
\bigl\|\f{\pa_z^m\vv u_k}r\bigr\|
_{L^2_t(L^2)}^{\f1m}\Bigr)
D_0^{\f14-\f1{2m}}(t)D_1^\f14(t)\\
&\lesssim \bigl(N^{-\f{m+1}{m}\left(\f14-\delta\right)}
+N^{-\f1m\left(\f32-\delta-\eta\right)}\bigr)
\|\vv\al\|_{L^2}^{\f12-\f1m}
\|\p_z\vv\al\|_{L^2}^{\f12}
\sum_{j=0}^{m+1}
\|\pa_z^j\vv\al\|_{L^2}^{\f2m}\\
&\quad+N^{-\f1m(1-\eta)}\|\vv\al\|_{L^2}^{\f12-\f1m}
\|\p_z\vv\al\|_{L^2}^{\f12}
\|\pa_z^m\vv\al\|_{L^2}^{\f1m}\\
&\lesssim N^{-\f{m+1}{m}(\f14-\delta)}
\sum_{j=0}^{m+1}
\|\pa_z^j\vv\al\|_{L^2}^{\f{m+1}{m}}
+N^{-\f1m(1-\eta)}
\sum_{j=0}^{m+1}
\|\pa_z^j\vv\al\|_{L^2}
\Bigl(1+N^{-\f1m\left(\f12-\delta\right)}\sum_{j=0}^{m+1}
\|\pa_z^j\vv\al\|_{L^2}^{\f1m}\bigr)\\
&\lesssim\e^{\f{m+1}{m}}+\e\bigl(1
+N^{-\f1{4m}}\e^{\f1m}\bigr),
\end{align*}
and
\begin{align*}
&\Bigl(N^{\f{m-2}{m-1}\left(\eta-\f14\right)}
\bigl\|\f{\pa_z^m (\ur_0,\ut_0)}r\bigr\|_{L^2_t(L^2)}^{\f1{m-1}}
+N^{\eta-\delta}\sup_{k\in\N^+}
\bigl\|\f{\pa_z^m\vv u_k}r\bigr\|
_{L^2_t(L^2)}^{\f1{m-1}}\Bigr)
D_0^\f14(t) D_1^{\f14-\f1{2(m-1)}}(t)\\
&\lesssim N^{-\f{m}{m-1}\left(\f14-\delta\right)}
\sum_{j=0}^{m+1}
\|\pa_z^j\vv\al\|_{L^2}^{\f m{m-1}}
+N^{-\f{1-\eta}{m-1}}
\sum_{j=0}^{m+1}
\|\pa_z^j\vv\al\|_{L^2}
\Bigl(1+N^{-\f1{m-1}\left(\f12-\delta\right)}\sum_{j=0}^{m+1}
\|\pa_z^j\vv\al\|_{L^2}^{\f1{m-1}}\bigr)\\
&\lesssim\e^{\f{m}{m-1}}+N^{-\f{1-\eta}{m(m-1)}}
\e\bigl(1+N^{-\f1{4(m-1)}}\e^{\f1{m-1}}\bigr).
\end{align*}
By substituting the above  estimates into \eqref{ANScylE1} and \eqref{ANScylE2}, we deduce that for any $t\leq T^\star_2$
$$ D_j(t)\leq D_j(0)+C\e D_j(t),\quad\mbox{for}\ \ j\in\{0,1\}.$$
In particular, if we take $\e$ to be so small that
$C\e\leq\f12$, we infer
\begin{equation}\label{2.16}
D_j(t)\leq 2D_j(0),\quad\mbox{for}\ \ j\in\{0,1\}\andf t\leq T_2^\star,
\end{equation}
which contradicts with \eqref{2.15} unless $T^\star_2=\infty.$ This in turn shows that
 $T^\star_2=\infty$, and \eqref{2.16} holds for any $t>0$. This completes the proof Theorem \ref{thm2}.
\end{proof}

\section{Anisotropic type Sobolev inequalities}\label{sec3}

In this section, we shall present anisotropic type Sobolev inequalities in cylindrical coordinates.
These inequalities will be  the most crucial ingredients used in the proof of the main results in this paper, and they
might be used elsewhere for problems in curved coordinates.

Let $\cG_\h$ be the set of smooth enough functions defined in the unit ball $B_{\h}\eqdefa \bigl\{ (x_1,x_2)\in\R^2, \ x_1^2+x_2^2<1 \bigr\}$ of $\R_\h^2$,
which vanish on the boundary of  $B_{\h}$. Then it follows from   the classical Sobolev inequality that
\begin{equation}\label{3.1}
\|f\|_{L^p_\h}\lesssim\|f\|_{L^2_\h}^\f2p
\|\nablah f\|_{L^2_\h}^{1-\f2p},
\quad\forall\ f\in\cG_\h,~2\leq p<\infty.
\end{equation}
Here and in the rest of this section, we always denote $\|f\|_{L^p_\h}\eqdefa \Bigl(\int_{B_\h}|f(x_\h)|^p\,dx_\h\Bigr)^{\f1p}.$

While we get, by using 1-D Sobolev inequality and Minkowski's inequality, that
\begin{equation}\begin{split}\label{3.2}
\|f\|_{L^p_\h}&\lesssim
\bigl\|\|f(\cdot,x_2)\|_{L^2_{x_1}}^{\f12+\f1p}
\|\pa_1 f(\cdot,x_2)\|_{L^2_{x_1}}^{\f12-\f1p}
\bigr\|_{L^p_{x_2}}\\
&\lesssim\|f\|_{L^2_{x_1}
(L^{\f{4+2p}{6-p}}_{x_2})}^{\f12+\f1p}
\|\pa_1 f\|_{L^2_{x_1}(L^2_{x_2})}^{\f12-\f1p}\\
&\lesssim\|f\|_{L^2}^{\f2p}
\|\pa_2 f\|_{L^2}^{\f12-\f1p}
\|\pa_1 f\|_{L^2}^{\f12-\f1p},
\quad\forall\ f\in\cG_\h,~2\leq p\leq6.
\end{split}\end{equation}

Compared with \eqref{3.1}, we call \eqref{3.2} to be  anisotropic type Sobolev  inequality. Then a natural question arises:
 does there exist some similar type of  anisotropic inequalities in the curved coordinates?
 In fact, for the polar coordinates in $\R_\h^2$, where $\nablah=e_r\pa_r+e_\th\f{\pa_\th}r$ with $e_r\eqdefa (\cos\theta,\sin\theta)$
  and $e_\theta\eqdefa (-\sin\theta,\cos\theta),$
  we have the following result:

\begin{lem}\label{lem3.1}
{\sl Let $f\in\cG_\h$ which satisfies
\begin{equation}\label{3.3}
\int_0^{2\pi}f(r,\th)\,d\th=0,
\quad\forall\ r\in\R^+.
\end{equation}
Then for any $2\leq p\leq6$, there holds
\begin{equation}\label{3.4}
\|f\|_{L^p_{\rm h}}
\lesssim\|f\|_{L^2_\h}^{\f2p}
\Bigl(\|\pa_r f\|_{L^2_\h}^{\f12-\f1p}
+\bigl\|\f{\pa_\th f}r\bigr\|
_{L^2_\h}^{\f12-\f1p}\Bigr)
\bigl\|\f{ \pa_\th f}r\bigr\|
_{L^2_\h}^{\f12-\f1p}.
\end{equation}
}\end{lem}

\begin{proof}
In view of  \eqref{3.3}, we get, by using 1-D Sobolev embedding inequality, that
\begin{equation}\begin{split}\label{3.5}
\|f\|_{L^p_\h}
&=\bigl\|r^{\f1p}\|f(r,\cdot)\|_{L^p(d\th)}\|_{L^p(dr)}\\
&\lesssim
\bigl\|r^{\f1p}\|f(r,\cdot)\|
_{L^2(d\th)}^{\f12+\f1p}
\|\pa_\th f(r,\cdot)\|_{L^2(d\th)}^{\f12-\f1p}\bigr\|_{L^p(dr)}\\
&= \Bigl\|\bigl\|r^{\f12}f(r,\cdot)\bigr\|
_{L^2(d\th)}^{\f12+\f1p}
\bigl\|r^{-\f12}\pa_\th f(r,\cdot)\bigr\|
_{L^2(d\th)}^{\f12-\f1p}\Bigr\|_{L^p(dr)}\\
&\lesssim\bigl\|r^{\f12}f\bigr\|
_{L^2\bigl(d\th;L^{\f{4+2p}{6-p}}(dr)\bigr)}^{\f12+\f1p}
\bigl\|r^{-\f12}\pa_\th f\bigr\|
_{L^2(drd\th)}^{\f12-\f1p}\\
&\lesssim\bigl\|r^{\f12}f\bigr\|_{L^2(drd\th)}^{\f2p}
\bigl\|\pa_r\bigl(r^{\f12}f\bigr)\bigr\|
_{L^2(drd\th)}^{\f12-\f1p}
\bigl\|r^{-\f12}\pa_\th f\bigr\|
_{L^2(drd\th)}^{\f12-\f1p}\\
&\lesssim\|f\|_{L^2_\h}^{\f2p}
\Bigl(\|\pa_r f\|_{L^2_\h}^{\f12-\f1p}
+\bigl\|\f{f}r\bigr\|
_{L^2_\h}^{\f12-\f1p}\Bigr)
\bigl\|\f{\pa_\th f}r\bigr\|
_{L^2_\h}^{\f12-\f1p}.
\end{split}\end{equation}
While it follows from  \eqref{3.3} and Poincar\'e's inequality that
$$\bigl\|\f{f}r\bigr\|_{L^2_\h}
\leq2\pi\bigl\|\f{\pa_\th f}r\bigr\|_{L^2_\h}.$$
By substituting the above inequality into \eqref{3.5}, we obtain  \eqref{3.4}.
\end{proof}

\begin{cor}
{\sl Let $g(r)\in\cG_\h$ be a radial function. Then one has
\begin{equation}\label{3.6}
\|g\|_{L^4_{\rm h}}
\lesssim\| g\|_{L^2_{\rm h}}^\f12
\bigl(\|\p_r g\|_{L^2_{\rm h}}^\f14
+\bigl\|\f{g}r\bigr\|_{L^2_{\rm h}}^\f14\bigr) \bigl\|\f{g}r\bigr\|_{L^2_{\rm h}}^\f14.
\end{equation}
}\end{cor}

\begin{proof}
 It is obvious that $g(r)\cos\th$ satisfies \eqref{3.3}. Then  we deduce  \eqref{3.6} by applying
  Lemma \ref{lem3.1} with $f(r,\th)=g(r)\cos\th$ and $p=4$.
\end{proof}

We mention that there is a restriction of $2\leq p\leq6$ in \eqref{3.2} and \eqref{3.4}.
Yet  for radial functions, we can use an alternative approach  to get rid of this restriction.

\begin{lem}\label{lem3.2}
{\sl Let $g(r)\in\cG_\h$ be a radial function. Then one has
\begin{equation}\label{3.7}
\|g\|_{L^p_{\rm h}}
\lesssim\|g\|_{L^2_\h}^{\f2p}
\Bigl(\|\pa_r g\|_{L^2_\h}^{\f12-\f1p}
+\bigl\|\f{g}r\bigr\|
_{L^2_\h}^{\f12-\f1p}\Bigr)
\bigl\|\f{g}r\bigr\|_{L^2_\h}^{\f12-\f1p},
\quad\forall\ 2\leq p<\infty.
\end{equation}
}\end{lem}

\begin{proof}
Let $\tau=r^{\f4{p+2}}$ and $h(\tau)=r^{\f2{p+2}}g(r)=\tau^\f12 g(\tau^{\f{p+2}4}),$ we  write
\begin{align*}
\|g\|_{L^p_\h}=C\Bigl(\int_0^1 \bigl|r^{\f2{p+2}}g(r)\bigr|^p\, d(r^{\f4{p+2}})\Bigr)^{\f1p}
=C\Bigl(\int_0^1|h(\tau)|^p\,d \tau\Bigr)^{\f1p}
\end{align*}
Then we get, by using 1-D Sobolev  inequality and the fact: $$h'(\tau)=\f12\tau^{-\f12}g(\tau^{\f{p+2}4})
+\f{p+2}4\tau^{\f p4}g'(\tau^{\f{p+2}4}),$$
that
\begin{align*}
\|g\|_{L^p_\h} &\lesssim \Bigl(\int_0^1|h'(\tau)|^2 \,d \tau \Bigr)^{\f14-\f1{2p}}
\Bigl(\int_0^1|h(\tau)|^2 \,d\tau\Bigr)^{\f14+\f1{2p}}\\
&\lesssim \Bigl(\int_0^1 \bigl|g(\tau^{\f{p+2}4})\bigr|^2 \,\f{d\tau}\tau
+\int_0^1\bigl|g'(\tau^{\f{p+2}4})\bigr|^2
\tau^{\f p2} \,d\tau\Bigr)^{\f14-\f1{2p}}
\Bigl(\int_0^1\bigl|g(\tau^{\f{p+2}4})\bigr|^2\tau \,d \tau\Bigr)^{\f14+\f1{2p}}\\
&\lesssim \Bigl(\int_0^1|g(r)|^2 \,\f{dr}r
+\int_0^1|g'(r)|^2 r\,dr\Bigr)^{\f14-\f1{2p}}
\Bigl(\int_0^1|g(r)|^2r \,dr\Bigr)^{\f1p}
\Bigl(\int_0^1|g(r)|^2\,\f{dr}r\Bigr)^{\f14-\f1{2p}}.
\end{align*}
This completes the proof of this lemma.
\end{proof}

\begin{rmk}
It is easy to verify that all the inequalities derived in this section works also for the whole space $\R_\h^2$.
\end{rmk}

We end this section with some  estimates of $\vv u_0$ and $\vv u_k$, which will be frequently used in the proof of Theorem \ref{thm1}.

\begin{lem}\label{lembddE}
{\sl Let $E_0(t)$ and $E_1(t)$ be defined by \eqref{def:E}, we have
\begin{equation}\label{lem3.3eq1}
\|\vv u_0\|_{L^4_t(L^4_{\rm h}(L^2_{\rm v}))} \lesssim N^{\eta-\f14}E_0^\f12(t), \quad
\|\p_z \vv u_0\|_{L^4_t(L^4_{\rm h}(L^2_{\rm v}))} \lesssim N^{\eta-\f14}E_1^\f12(t),
\end{equation}
and for any $k\geq 1$ that
\begin{equation}\label{lem3.3eq2}
\|\vv u_k\|_{L^4_t(L^4_{\rm h}(L^2_{\rm v}))} \lesssim k^{-\f54} N^{\eta(2-k)-\f14} E_0^\f12(t),
\quad \|\p_z\vv u_k\|_{L^4_t(L^4_{\rm h}(L^2_{\rm v}))} \lesssim k^{-\f54} N^{\eta(2-k)-\f14} E_1^\f12(t),
\end{equation}
and
\begin{subequations}\label{lem3.3eq3}
\begin{gather}
\label{S3eq1a}\|\vv u_0\|_{L^4_t(L^4_{\rm h}(L^\infty_{\rm v}))}+\|\wt\nabla \vv u_0\|_{L^2_t(L^2_{\rm h}(L^\infty_{\rm v}))}+\bigl\|\f{(\ur_0,\ut_0)}r\bigr\|_{L^2_t(L^2_{\rm h}(L^\infty_{\rm v}))}\lesssim N^{\eta-\f14} E_0^\f14(t)E_1^\f14(t),\\
\label{S3eq1b}\|\vv u_k\|_{L^4_t(L^4_{\rm h}(L^\infty_{\rm v}))} \lesssim k^{-\f54}N^{\eta(2-k)-\f14} E_0^\f14(t)E_1^\f14(t),\\
\label{S3eq1c}\|\wt\nabla \vv u_k\|_{L^2_t(L^2_{\rm h}(L^\infty_{\rm v}))}+kN\bigl\|\f{\vv u_k}r\bigr\|_{L^2_t(L^2_{\rm h}(L^\infty_{\rm v}))}
\lesssim k^{-1}N^{\eta(2-k)} E_0^\f14(t)E_1^\f14(t).
 \end{gather}
\end{subequations}
}\end{lem}

\begin{proof} Observing that $\vv u_0=u_0^r\vv e_r+u_0^\theta\vv e_\theta+u_0^z\vv e_z,$
we deduce from  the classical Sobolev inequality \eqref{3.1} that
\beq\label{S3eq45}
\begin{split}
\|\vv u_0\|_{L^4_t(L^4_{\rm h}(L^2_{\rm v}))}\leq &\|\vv u_0\|_{L^4_t(L^2_{\rm v}(L^4_{\rm h}))}\\
\lesssim &
\|\vv u_0\|_{L^\infty_t(L^2)}^\f12
\Bigl(\|\pa_r\vv u_0\|_{L^2_t(L^2)}^{\f12}
+\bigl\|\f{(\ur_0,\ut_0)}r
\bigr\|_{L^2_t(L^2)}^{\f12}\Bigr).
\end{split}\eeq
While it follows from the anisotropic Sobolev inequality \eqref{3.6} that
\begin{equation}\label{lem3.3eq4}
\|\vv u_k\|_{L^4_t(L^4_{\rm h}(L^2_{\rm v}))}
\lesssim (kN)^{-\f14}\|\vv u_k\|_{L^\infty_t(L^2)}^\f12
\bigl\|\f{kN}r \vv u_k\bigr\|_{L^2_t(L^2)}^\f14
\Bigl(\|\p_r \vv u_k\|_{L^2_t(L^2)}^\f14
+\bigl\|\f{kN}r \vv u_k\bigr\|_{L^2_t(L^2)}^\f14\Bigr),
\end{equation} from which,  \eqref{def:E} and \eqref{S3eq45}, we derive
 the first inequalities of \eqref{lem3.3eq1} and \eqref{lem3.3eq2}. And the second inequalities of \eqref{lem3.3eq1} and \eqref{lem3.3eq2} follow along the same line.

On the other hand, by using the interpolation inequality
\beq\label{S3eq2} \|f\|_{L^\infty_\v}\lesssim
\|f\|_{L^2_\v}^{\f12}\|\pa_zf\|_{L^2_\v}^{\f12},\eeq
 \eqref{lem3.3eq1} and \eqref{lem3.3eq2}, we  deduce \eqref{lem3.3eq3}.
This completes the proof of Lemma \ref{lembddE}.
\end{proof}

\begin{rmk}
If instead of using the anisotropic Sobolev inequality \eqref{3.6}  in the derivation of \eqref{lem3.3eq4},
 we use the isotropic inequality,  we find
$$\|\p_z^j\vv u_k\|_{L^4_t(L^4_{\rm h}(L^2_{\rm v}))} \lesssim k^{-1} N^{\eta(2-k)} E_j^\f12(t)\quad\mbox{for}\ j=0,1.$$
 It is easy to observe that the right-hand side of \eqref{lem3.3eq2} has an additional decay in $k,$ $(kN)^{-\f14}$,
  which turns out to be  crucial in the proof of Theorem \ref{thm1}. This reflects the subtlety of the anisotropic inequalities.
\end{rmk}

\section{Proof of Proposition \ref{main prop}}\label{sec4}
The goal of this section is to present the proof of Proposition \ref{main prop}. The main ingredients will be
  the following Lemmas \ref{lem4.1}-\ref{lem4.4}, which are concerned with the $L^2$-estimates of $\vv u_0,\,\pa_z\vv u_0,\,\vv u_k$ and $\pa_z\vv u_k$ respectively.

\begin{lem}\label{lem4.1}
{\sl Under the assumptions of Proposition \ref{main prop}, for any $t\leq T$, we have
\begin{equation}\begin{aligned}\label{S4eq1}
& \|\vv u_0\|_{L^\infty_t(L^2)}^2
+2\| \wt \nabla \vv u_0\|_{L^2_t(L^2)}^2
+2\bigl\|\f{(\ur_0,\ut_0)}r\bigr\|_{L^2_t(L^2)}^2
\lesssim N^{3(\eta-\f14)} E_0^\f54(t)E_1^\f14(t).
\end{aligned}\end{equation}
}
\end{lem}

\begin{proof}
We first get, by using integration by parts, the homogeneous Dirichlet boundary condition for $u_0$ and $\dive u_0=0$, that
$\bigl( u_0\cdot \wt\nabla \vv u_0
\big|\vv u_0\bigr)=0$ and
$$\int_{\wt\Omega} \bigl( \ur_0 \p_r P_0+\uz_0 \p_z P_0\bigr)\,rdrdz=-\int_{\wt\Omega} \bigl( \pa_r\ur_{0}+\frac{\ur_{0}}r
+\pa_z\uz_{0}\bigr)P_0\,rdrdz=0,
$$
where ${\wt\Omega}\eqdef \{(r,z)\in[0,1[\times\R\}$.
Then by taking $L^2$ inner product of \eqref{eqtu0} for $\nu=1$ with $\vv u_0$, we find
\begin{equation}\begin{aligned}\label{S4eq2}
\f12\f{d}{dt}&\|\vv u_0\|_{L^2}^2
+\| \wt \nabla \vv u_0\|_{L^2}^2
+\bigl\|\f{(\ur_0,\ut_0)}r \bigr\|_{L^2}^2\\
&=-\f12\sum_{k=1}^\infty\Bigl( \bigl(\wt u_k\cdot \wt\nabla-\f{kN}r \vt_k\bigr) u_k  +\bigl(\wt v_k\cdot \wt\nabla+\f{kN}r \ut_k\bigr) v_k\Big|\vv u_0\Bigr)\\
&\qquad+\f12\sum_{k=1}^\infty\Bigl(  (\ut_k)^2+{(\vt_{k})^2}  \Big|\f{\ur_0}r\Bigr)
-\f12\sum_{k=1}^\infty\Bigl( \ur_k\ut_k+\vr_k\vt_k \Big|\f{\ut_0}r\Bigr),
\end{aligned}\end{equation}
where $u_k\eqdef(\ur_k,\ut_k,\uz_k),
~v_k\eqdef(\vr_k,\vt_k,\vz_k)$.

It is easy to observe from the fourth equations of
 \eqref{eqtuk} and \eqref{eqtvk} that
\begin{equation}\label{S4eq2a}
\int_{\wt\Omega} \bigl(\wt u_k\cdot \wt\nabla-\f{kN}r \vt_k\bigr)f \cdot g\,rdrdz = -\int_{\wt\Omega} f \cdot (\wt u_k\cdot \wt\nabla g)\,rdrdz
\end{equation}
and
\begin{equation}\label{S4eq2b}
\int_{\wt\Omega} \bigl(\wt v_k\cdot \wt\nabla+\f{kN}r \ut_k\bigr)f \cdot g\, rdrdz= -\int_{\wt\Omega} f \cdot (\wt v_k\cdot \wt\nabla g)\, rdrdz.
\end{equation}
By applying \eqref{S4eq2a} and \eqref{S4eq2b} to the second line of \eqref{S4eq2}, we obtain
\begin{align*}
\Bigl( \bigl(\wt u_k\cdot \wt\nabla-\f{kN}r \vt_k\bigr) u_k  +\bigl(\wt v_k\cdot \wt\nabla
+\f{kN}r \ut_k\bigr) v_k \Big|\vv u_0\Bigr)
=-\bigl(u_k\big|\wt u_k\cdot \wt\nabla \vv u_0\bigr)
-\bigl( v_k\big|\wt v_k\cdot \wt\nabla \vv u_0\bigr).
\end{align*}
Then we get,
by substituting the above equality into \eqref{S4eq2} and using H\"older's inequality, that
\begin{equation}\begin{aligned}\label{S4eq3}
\|\vv u_0\|_{L^\infty_t(L^2)}^2
&+2\| \wt \nabla \vv u_0\|_{L^2_t(L^2)}^2
+2\|\f{(\ur_0,\ut_0)}r \|_{L^2_t(L^2)}^2 \\
&\lesssim\sum_{k=1}^\infty
\int_0^t\int_{\wt\Omega} |\vv u_k|^2 \bigl(|\wt\nabla \vv u_0|+\bigl|\f{(\ur_0,\ut_0)}r\bigr| \bigr)\,rdrdzdt'\\
&\lesssim\sum_{k=1}^\infty
\|\vv u_k\|_{L^4_t(L^4_{\rm h}(L^\oo_{\rm v}))}\|\vv u_k\|_{L^4_t(L^4_{\rm h}(L^2_{\rm v}))} \bigl(\|\wt\nabla \vv u_0\|_{L^2_t(L^2)}
+\bigl\|\f{(\ur_0,\ut_0)}r\bigr\|_{L^2_t(L^2)}\bigr),
\end{aligned}\end{equation}
 from which and Lemma \ref{lembddE}, we infer
\begin{align*}
\|\vv u_0\|_{L^\infty_t(L^2)}^2
&+2\| \wt \nabla \vv u_0\|_{L^2_t(L^2)}^2
+2\|\f{(\ur_0,\ut_0)}r \|_{L^2_t(L^2)}^2 \\
&\lesssim N^{3(\eta-\f14)} E_0^\f54(t)E_1^\f14(t)
\sum_{k=1}^\infty k^{-\f52}\lesssim N^{3(\eta-\f14)}
E_0^\f54(t)E_1^\f14(t).
\end{align*}
This completes the proof of Lemma \ref{lem4.1}.
\end{proof}

\begin{lem}\label{S4prop2}
{\sl Under the assumptions of Proposition \ref{main prop}, for any $t\leq T$, we have
\begin{equation}\begin{aligned}\label{S4eq11}
& \|\p_z\vv u_0\|_{L^\infty_t(L^2)}^2
+2\|\wt\nabla\pa_z\vv u_0\|_{L^2_t(L^2)}^2
+2\bigl\|\f{(\pa_z\ur_0,\pa_z\ut_0)}r \bigr\|_{L^2_t(L^2)}^2
\lesssim N^{3(\eta-\f14)} E_0^\f14(t)E_1^\f54(t).
\end{aligned}\end{equation}
}\end{lem}

\begin{proof} Similar to the derivation of \eqref{S4eq2},  by taking $L^2$ inner product of the system \eqref{eqtu0} for $\nu=1$ with $-\p_z^2 \vv u_0$
and using integration by parts, we obtain
\begin{equation}\begin{aligned}\label{S4eq12}
\f12\f{d}{dt}\|\p_z&\vv u_0\|_{L^2}^2
+\| \wt \nabla \p_z\vv u_0\|_{L^2}^2
+\bigl\|\f{(\p_z\ur_0,\p_z\ut_0)}r\bigr\|_{L^2}^2\\
=&\bigl( u_0\cdot \wt\nabla \vv u_0
\big|\p_z^2\vv u_0\bigr)
-\bigl(\f{(\ut_0)^2}r \big|\p_z^2\ur_0\bigr)
+\bigl(\f{\ut_0\ur_0}r\big|\p_z^2\ut_0\bigr)\\
&+\f12\sum_{k=1}^\oo\Bigl( \bigl(\wt u_k\cdot \wt\nabla-\f{kN}r \vt_k\bigr) u_k +\bigl(\wt v_k\cdot \wt\nabla+\f{kN}r \ut_k\bigr) v_k\Big|\p_z^2\vv u_0\Bigr)\\
&-\f12\sum_{k=1}^\infty\Bigl(  (\ut_k)^2+{(\vt_{k})^2}  \Big|\f{\p_z^2\ur_0}r\Bigr)
+\f12\sum_{k=1}^\infty\Bigl( \ur_k\ut_k+\vr_k\vt_k \Big|\f{\p_z^2\ut_0}r\Bigr).
\end{aligned}\end{equation}
By using integration by parts and  \eqref{S4eq2a}, \eqref{S4eq2b}, we find
\begin{align*}
&\Bigl( \bigl(\wt u_k\cdot \wt\nabla-\f{kN}r \vt_k\bigr) u_k +\bigl(\wt v_k\cdot \wt\nabla+\f{kN}r \ut_k\bigr) v_k\Big|\p_z^2\vv u_0\Bigr)\\
&=-\bigl( u_k \big| \wt u_k\cdot \wt\nabla\p_z^2\vv u_0\bigr)
-\bigl(v_k \big| \wt v_k\cdot \wt\nabla\p_z^2\vv u_0\bigr)\\
&=\bigl(\p_z u_k \big| \wt u_k\cdot \wt\nabla\p_z\vv u_0\bigr)_{L^2}+\bigl(u_k | \p_z\wt u_k\cdot \wt\nabla\p_z\vv u_0\bigr)
+\bigl(\p_z v_k\big| \wt v_k\cdot \wt\nabla\p_z\vv u_0\bigr)
+\bigl(v_k \big| \p_z\wt v_k\cdot \wt\nabla\p_z\vv u_0\bigr),
\end{align*}
and
\begin{align*}
-\Bigl((\ut_k)^2+(\vt_{k})^2  \Big|\f{\p_z^2\ur_0}r\Bigr)
&+\Bigl(\ur_k\ut_k+\vr_k\vt_k \Big|\f{\p_z^2\ut_0}r\Bigr)\\
&=2\Bigl(\ut_k\pa_z\ut_k+\vt_k\pa_z\vt_k  \Big|\f{\p_z\ur_0}r\Bigr)
-\Bigl(\pa_z\bigl(\ur_k\ut_k+\vr_k\vt_k\bigr) \Big|\f{\p_z\ut_0}r\Bigr).
\end{align*}
By substituting the above inequalities  into \eqref{S4eq12}
and then integrating the resulting equation over $[0,t],$ we infer
\begin{equation}\label{S4eq13}
\f12\|\p_z \vv u_0\|_{L^\infty_t(L^2)}^2
+\| \wt \nabla \p_z\vv u_0\|_{L^2_t(L^2)}^2
+\bigl\|\f{(\p_z\ur_0,\p_z\ut_0)}r
\bigr\|_{L^2_t(L^2)}^2\leq\sum_{j=1}^3I_j(t),
\end{equation}
where
\beq\begin{split}\label{defI}
I_1(t)&\eqdef\int_0^t \Bigl|\bigl( u_0\cdot \wt\nabla \vv u_0 \big|\p_z^2\vv u_0\bigr)\Bigr|\,dt'
=\int_0^t \bigl|\int_{\wt\Omega}(\p_z u_0\cdot \wt\nabla\p_z\vv u_0) \cdot \vv u_0 \,rdrdz\bigr|\,dt'\\
&\leq\|\vv u_0\|_{L^4_t(L^4_{\rm h}(L^\oo_{\rm v}))}\|\p_z\vv u_0\|_{L^4_t(L^4_{\rm h}(L^2_{\rm v}))} \|\wt\nabla \p_z\vv u_0\|_{L^2_t(L^2)},\\
I_2(t)&\eqdef\int_0^t \Bigl|\bigl(\f{(\ut_0)^2}r \big|\p_z^2\ur_0\bigr)-\bigl(\f{\ut_0\ur_0}r \big|\p_z^2\ut_0\bigr)\Bigr|\,dt'\\
&\leq\|\vv u_0\|_{L^4_t(L^4_{\rm h}(L^\oo_{\rm v}))}\|\p_z\vv u_0\|_{L^4_t(L^4_{\rm h}(L^2_{\rm v}))} \bigl\| \f{(\p_z \ur_0,\p_z \ut_0)}r\bigr\|_{L^2_t(L^2)},\\
I_3(t)&\eqdef \sum_{k=1}^\infty\int_0^t\int_{\wt\Omega}|\vv u_k||\p_z \vv u_k| \Bigl(|\wt\nabla \p_z\vv u_0|+\bigl|\f{(\p_z\ur_0,\p_z\ut_0)}r\bigr| \Bigr)\,rdrdzdt'\\
&\leq\sum_{k=1}^\infty
\|\vv u_k\|_{L^4_t(L^4_\h(L^\infty_\v))}
\|\p_z\vv u_k\|_{L^4_t(L^4_\h(L^2_\v))}
\Bigl\|\Bigl(\wt\nabla \p_z\vv u_0,
\f{(\p_z\ur_0,\p_z\ut_0)}r\Bigr)\Bigr\|
_{L^2_t(L^2)},
\end{split}\eeq
from which and Lemma \ref{lembddE}, we deduce that
\begin{equation}\label{S4eq14}
I_1(t)\lesssim N^{3(\eta-\f14)} E_0^\f14(t) E_1^\f54(t),\quad
I_2(t)\lesssim N^{3(\eta-\f14)}E_0^\f14(t) E_1^\f54(t),
\end{equation}
and
\begin{equation}\label{S4eq16}
I_3(t)\lesssim N^{\eta-\f14}
\sum_{k=1}^\infty
k^{-\f52}N^{2\eta(2-k)-\f12}E_0^\f14(t)E_1^\f54(t)
\lesssim N^{3(\eta-\f14)}E_0^\f14(t) E_1^\f54(t).
\end{equation}

By inserting the inequalities \eqref{S4eq14} and \eqref{S4eq16} into \eqref{S4eq13}, we achieve  \eqref{S4eq11}. This completes the proof of Lemma \ref{S4prop2}.
\end{proof}

\begin{lem}\label{lem4.3}
{\sl Under the assumptions of Proposition \ref{main prop}, for any $t\leq T$ and $k\geq 1$, we have
\begin{equation}\begin{aligned}\label{S4eq21}
\|\vv u_k\|_{L^\infty_t(L^2)}^2
+2\|\wt \nabla \vv u_k\|_{L^2_t(L^2)}^2
&+\f{k^2N^2}2\bigl\|\f{\vv u_k}r \bigr\|_{L^2_t(L^2)}^2\\
&\leq \|\vv u_k(0)\|_{L^2}^2
+Ck^{-\f94}N^{2\eta(2-k)+2(\eta-\f14)}
E_0^\f54(t) E_1^\f14(t).
\end{aligned}\end{equation}
}
\end{lem}

\begin{proof} Due to $\vv u_k|_{r=1}=0,$ we get, by using integration by parts and the fourth equations of
\eqref{eqtuk} and \eqref{eqtvk}, that
\begin{align*}
\int_{\wt\Omega}\bigl(\ur_k \p_r P_k
&-\f{kN}r \vt_k P_k+\uz_k \p_z P_k\bigr)r\,dr\,dz\\
&=-\int_{\wt\Omega}\bigl( \pa_r\ur_{k}+\frac{\ur_k}r+\f{kN}r \vt_k+\pa_z\uz_{k}\bigr)P_k\,rdrdz=0,
\end{align*}
and
\begin{align*}
\int_{\wt\Omega}\bigl( \vr_k \p_r Q_k
& +\f{kN}r \ut_k Q_k+\vz_k \p_z Q_k\bigr)r\,dr\,dz\\
&=-\int_{\wt\Omega}\bigl( \pa_r\vr_{k}+\frac{\vr_k}r-\f{kN}r \ut_k+\pa_z\vz_{k}\bigr)Q_k\,rdrdz=0.
\end{align*}
As a result, for $\nu=1,$  by taking $L^2$ inner product of \eqref{eqtuk} with $u_k$ and \eqref{eqtvk} with $v_k$, we obtain
\begin{equation}\begin{aligned}\label{S4eq22}
\f12\f{d}{dt}&\|\vv u_k\|_{L^2}^2
+\|\wt \nabla \vv u_k\|_{L^2}^2
+k^2N^2\|\f{\vv u_k}r\|_{L^2}^2
+\|\f{(\ur_k,\vt_k,\vr_k,\ut_k)}r \|_{L^2}^2\\
=&4kN \bigl(\f{\vr_k}r \big|\f{\ut_k}r \bigr)
-4kN \bigl(\f{\ur_k}r \big|\f{\vt_k}r \bigr)
-\bigl(\wt u_k\cdot \wt\nabla \vv u_0 \big|u_k\bigr)
-\bigl(\wt v_k\cdot \wt\nabla \vv u_0
\big|v_k\bigr)
+2\bigl(\f{\ut_0\ut_k}r\big|\ur_k\bigr)\\
&-\bigl(\f{\ur_0\vt_k+\ut_0\vr_k}r\big|\vt_k\bigr)
+2\bigl(\f{\ut_0\vt_k}r\big|\vr_k\bigr)
-\bigl(\f{\ur_0\ut_k+\ut_0\ur_k}r\big|\ut_k\bigr) +\bigl( (F_k,G_k) \big|\vv u_k\bigr).
\end{aligned}\end{equation}

Observing that $2kN\leq \f34 k^2N^2+1$
for any $N\geq2$, we infer
$$4kN\Bigl| \bigl(\f{\ur_k}r |\f{\vt_k}r \bigr)
-\bigl(\f{\vr_k}r |\f{\ut_k}r \bigr)\Bigr|
\leq\bigl(\f{3k^2N^2}4+1\bigr)
\bigl\|\f{(\ur_k,\vt_k,\vr_k,\ut_k)}r\bigr\|_{L^2}^2.$$
By substituting the above inequality into \eqref{S4eq22} and integrating the resulting inequality over $[0,t],$ we achieve
\begin{equation}\label{S4eq23}
\|\vv u_k\|_{L^\infty_t(L^2)}^2
+2\| \wt \nabla \vv u_k\|_{L^2_t(L^2)}^2
+\f{k^2N^2}2\|\f{\vv u_k}r \|_{L^2_t(L^2)}^2
\leq \|\vv u_k(0)\|_{L^2}^2
+CII_1(t)+CII_2(t),
\end{equation}
where
\begin{equation}\begin{split}\label{defII}
&II_1(t)\eqdef\|\vv u_k\|_{L^4_t(L^4_\h(L^2_\v))}^2
\Bigl(\|\wt\nabla \vv u_0\|
_{L^2_t(L^2_\h(L^\infty_\v))}
+\bigl\|\f{(\ur_0,\ut_0)}r\bigr\|
_{L^2_t(L^2_\h(L^\infty_\v))}\Bigr),\\
&II_2(t)\eqdef\int_0^t \Bigl|\bigl( (F_k,G_k)
\big|\vv u_k\bigr)\Bigr|\,dt'.
\end{split}\end{equation}

It follows from
 Lemma \ref{lembddE} that
\begin{equation}\label{S4eq24}
II_1(t)\lesssim k^{-2}N^{2\eta(2-k)}(kN)^{-\f12}N^{\eta-\f14}
E_0^\f54(t)E_1^\f14(t).
\end{equation}
While the estimate of $II_2(t)$ relies on the following inequality:
\begin{equation}\label{4.16}
\bigl|\bigl( (F_k,G_k) \big|\vv u_k\bigr)\bigr|
\lesssim\sum_{\substack{|k_1-k_2|=k\\
\text{ or }k_1+k_2=k}}
\int_{\wt\Omega}|\vv u_{k_1}||\vv u_{k_2}| \bigl(|\wt \nabla \vv u_k|
+kN\bigl|\f{\vv u_k}r\bigr| \bigr)\,rdrdz,
\end{equation}
the proof of which  will be postponed in Appendix \ref{appendixB}.

By virtue of \eqref{4.16}, we get, by using  Lemma \ref{lembddE}, that
\begin{equation}\begin{split}\label{4.17}
II_2(t)&\lesssim\sum_{\substack{
|k_1-k_2|=k\\
\text{ or }k_1+k_2=k}}
\|\vv u_{k_1}\|_{L^4_t(L^4_{\rm h}(L^\oo_{\rm v}))}\|\vv u_{k_2}\|_{L^4_t(L^4_{\rm h}(L^2_{\rm v}))} \bigl(\|\wt\nabla \vv u_k\|_{L^2_t(L^2)}
+kN\bigl\|\f{\vv u_k}r\bigr\|_{L^2_t(L^2)}\bigr)\\
&\lesssim k^{-1}N^{\eta(2-k)} E_0^\f54(t)E_1^\f14(t)\sum_{\substack{
|k_1-k_2|=k\\
\text{ or }k_1+k_2=k}}
 k_1^{-\f54}N^{\eta(2-k_1)-\f14}
 k_2^{-\f54}N^{\eta(2-k_2)-\f14}.
\end{split}\end{equation}
Notice that when $k_1+k_2=k$, there holds
\begin{align*}
\sum_{k_1+k_2=k}k_1^{-\f54}N^{\eta(2-k_1)-\f14}
k_2^{-\f54}N^{\eta(2-k_2)-\f14}
&=N^{\eta(2-k)+2(\eta-\f14)}\sum_{k_1+k_2=k}
k_1^{-\f54}k_2^{-\f54}\\
&\lesssim k^{-\f54}N^{\eta(2-k)+2(\eta-\f14)}.
\end{align*}
Whereas when $|k_1-k_2|=k$, there must hold $\max\{k_1,k_2\}\geq k+1$, so that one has
\begin{align*}
\sum_{|k_1-k_2|=k}k_1^{-\f54}N^{\eta(2-k_1)-\f14}
k_2^{-\f54}N^{\eta(2-k_2)-\f14}
&\leq (k+1)^{-\f54} N^{\eta(1-k)-\f14}
\sum_{j=1}^\infty j^{-\f54}N^{\eta(2-j)-\f14}\\
&\lesssim k^{-\f54} N^{\eta(2-k)-\f12}.
\end{align*}
As a consequence, we arrive at
\begin{equation}\label{S4eq25a}
\sum_{\substack{
|k_1-k_2|=k\\
\text{ or }k_1+k_2=k}}
 k_1^{-\f54}N^{\eta(2-k_1)-\f14}
 k_2^{-\f54}N^{\eta(2-k_2)-\f14}
\lesssim k^{-\f54}N^{\eta(2-k)+2(\eta-\f14)}.
\end{equation}

By inserting the inequality \eqref{S4eq25a} into \eqref{4.17},
we conclude that
\begin{equation}\label{S4eq25}
II_2(t)\lesssim k^{-\f94}
N^{2\eta(2-k)+2(\eta-\f14)} E_0^\f54(t)E_1^\f14(t).
\end{equation}

Substituting \eqref{S4eq24} and \eqref{S4eq25} into \eqref{S4eq23} leads to \eqref{S4eq21}, which completes the proof of Lemma \ref{lem4.3}.
\end{proof}

\begin{lem}\label{lem4.4}
{\sl Under the assumptions of Proposition \ref{main prop}, for any $t\leq T$ and $k\geq 1$, we have
\begin{equation}\begin{aligned}\label{S4eq31}
\|\p_z\vv u_k\|_{L^\infty_t(L^2)}^2
+2\| \wt \nabla \p_z\vv u_k&\|_{L^2_t(L^2)}^2
+\f{k^2N^2}2\bigl\|\f{\p_z\vv u_k}r \bigr\|_{L^2_t(L^2)}^2\\
&\leq\|\p_z\vv u_k(0)\|_{L^2}^2 +Ck^{-\f94}N^{2\eta(2-k)+2(\eta-\f14)}
E_0^\f14(t)E_1^\f54(t).
\end{aligned}\end{equation}
}
\end{lem}

\begin{proof} Similar to the derivation of \eqref{S4eq22}, for $\nu=1,$  by taking $L^2$ inner product of \eqref{eqtuk} with  $-\p_z^2u_k$
and \eqref{eqtvk} with $-\p_z^2v_k$, and using integration by parts, we find
\begin{equation}\begin{split}\label{4.21}
\f12&\f{d}{dt}\|\p_z\vv u_k\|_{L^2}^2 +\| \wt \nabla \p_z\vv u_k\|_{L^2}^2+k^2N^2\bigl\|\f{\p_z\vv u_k}r\bigr\|_{L^2}^2
+\bigl\|\f{(\p_z\ur_k,\p_z\vt_k,\p_z\vr_k,\p_z\ut_k)}r \bigr\|_{L^2}^2\\
=&4kN\bigl(\f{\p_z\vr_k}r\big|\f{\p_z\ut_k}r\bigr)
-4kN\bigl(\f{\p_z\ur_k}r\big|\f{\p_z\vt_k}r\bigr)
-\bigl(\pa_z u_0\cdot \wt\nabla \vv u_k \big|\p_z\vv u_k\bigr)\\
&-\bigl(\pa_z(\wt u_k\cdot \wt\nabla \vv u_0)
\big|\p_z u_k\bigr)
-\bigl(\pa_z(\wt v_k\cdot \wt\nabla \vv u_0) \big|\p_zv_k\bigr)
+2\bigl(\pa_z\f{\ut_0\ut_k}r\big|\p_z\ur_k\bigr)\\
&-\bigl(\pa_z\f{\ur_0\vt_k+\ut_0\vr_k}r
\big|\p_z\vt_k\bigr)
+2\bigl(\pa_z\f{\ut_0\vt_k}r|\p_z\vr_k\bigr)
-\bigl(\pa_z\f{\ur_0\ut_k+\ut_0\ur_k}r
\big|\p_z\ut_k\bigr)\\
&+kN\bigl(\f{(-\vr_k,\ut_k,-\vz_k,
\ur_k,-\vt_k,\uz_k)}r\pa_z\ut_0
\big|\p_z\vv u_k \bigr)
+\bigl(\pa_z(F_k,G_k) \big|\p_z\vv u_k\bigr).
\end{split}\end{equation}

Along the  same line to the derivation of \eqref{S4eq23}, we observe that for $N\geq2$
$$4kN\Bigl|\bigl(\f{\p_z\vr_k}r\big|\f{\p_z\ut_k}r\bigr)
-\bigl(\f{\p_z\ur_k}r\big|\f{\p_z\vt_k}r\bigr)\Bigr|
\leq\bigl(\f{3k^2N^2}4+1\bigr)
\bigl\|\f{(\p_z\ur_k,\p_z\vt_k,\p_z\vr_k,
\p_z\ut_k)}r\bigr\|_{L^2}^2.$$
By substituting the above inequality  into \eqref{4.21} and integrating the resulting inequality over $[0,t],$ we obtain
\begin{equation}\label{4.22}
\|\p_z\vv u_k\|_{L^\infty_t(L^2)}^2
+2\| \wt \nabla \p_z\vv u_k\|_{L^2_t(L^2)}^2
+\f{k^2N^2}2\|\f{\p_z\vv u_k}r \|_{L^2_t(L^2)}^2
\leq \|\pa_z\vv u_k(0)\|_{L^2}^2
+C\sum_{j=1}^3 III_j(t),
\end{equation}
where
\begin{align*}
&III_1(t)\eqdef\|\p_z\vv u_0\|_{L^4_t(L^4_{\rm h}(L^2_{\rm v}))}\|\p_z\vv u_k\|_{L^4_t(L^4_{\rm h}(L^2_{\rm v}))} \bigl(\|\wt\nabla \vv u_k\|_{L^2_t(L^2_{\rm h}(L^\oo_{\rm v}))}+kN\bigl\|\f{ \vv u_k}r\bigr\|_{L^2_t(L^2_{\rm h}(L^\oo_{\rm v}))}\bigr)\\
&III_2(t)\eqdef\|\p_z\vv u_k\|_{L^4_t(L^4_\h(L^2_\v))}
\|\vv u_k\|_{L^4_t(L^4_\h(L^\infty_\v))}
\bigl(\|\wt\nabla \p_z\vv u_0\|_{L^2_t(L^2)}
+\bigl\|\f{(\p_z\ur_0,\p_z\ut_0)}r
\bigr\|_{L^2_t(L^2)}\bigr)\\
&\qquad\qquad+\|\p_z\vv u_k\|_{L^4_t(L^4_{\rm h}(L^2_{\rm v}))}^2 \bigl(\|\wt\nabla \vv u_0\|_{L^2_t(L^2_{\rm h}(L^\oo_{\rm v}))}+\bigl\|\f{(\ur_0,\ut_0)}r\bigr\|_{L^2_t(L^2_{\rm h}(L^\oo_{\rm v}))}\bigr)\\
&III_3(t)\eqdef\int_0^t \Bigl|\bigl(\pa_z(F_k,G_k)
\big|\pa_z\vv u_k\bigr)\Bigr|\,dt'.
\end{align*}

It follows from  Lemma \ref{lembddE}  that
\begin{equation}\begin{split}\label{S4eq34}
&III_1(t)\lesssim
k^{-\f94}N^{2\eta(2-k)-\f14}N^{\eta-\f14}
E_0^\f14(t)E_1^\f54(t) \andf\\
&III_2(t)\lesssim k^{-\f52}
N^{2\eta(2-k)-\f12}N^{\eta-\f14}E_0^\f14(t)E_1^\f54(t).
\end{split}\end{equation}

While the estimate of $III_3(t)$ relies on the following inequality:
\begin{equation}\label{4.24}
\bigl|\bigl( \p_z(F_k,G_k) \big|\p_z\vv u_k\bigr)\bigr|
\lesssim\sum_{\substack{|k_1-k_2|=k\\
\text{ or }k_1+k_2=k}}
\int_{\wt\Omega}|\p_z\vv u_{k_1}||\vv u_{k_2}|
 \bigl( |\wt \nabla \p_z\vv u_k|+kN\bigl|\f{\p_z \vv u_k}r\bigr| \bigr)\,rdrdz.
\end{equation}
again the  proof of which will be postponed in Appendix \ref{appendixB}.

In view of  \eqref{4.24}, we deduce from Lemma \ref{lembddE}  that
\begin{align*}
III_3(t)&\lesssim\sum_{\substack{
|k_1-k_2|=k\\
\text{ or }k_1+k_2=k}}
\|\p_z \vv u_{k_1}\|_{L^4_t(L^4_{\rm h}(L^2_{\rm v}))}\|\vv u_{k_2}\|_{L^4_t(L^4_{\rm h}(L^\oo_{\rm v}))}\bigl(\|\wt\nabla \p_z \vv u_k\|_{L^2_t(L^2)}+kN\bigl\|\f{\p_z\vv u_k}r\bigr\|_{L^2_t(L^2)}\bigr)\\
&\lesssim k^{-1}N^{\eta(2-k)} E_0^\f14(t)E_1^\f54(t)\sum_{\substack{
|k_1-k_2|=k\\
\text{ or }k_1+k_2=k}}
 k_1^{-\f54}N^{\eta(2-k_1)-\f14}
 k_2^{-\f54}N^{\eta(2-k_2)-\f14},
\end{align*}
from which and \eqref{S4eq25a}, we infer
\begin{equation}\label{S4eq38}
\begin{aligned}
&III_3(t)
\leq  Ck^{-\f94}N^{2\eta(2-k)+2(\eta-\f14)} E_0^\f14(t)E_1^\f54(t).
\end{aligned}
\end{equation}

Substituting the inequalities \eqref{S4eq34} and \eqref{S4eq38} into \eqref{4.22} leads to \eqref{S4eq31},
which completes the proof of Lemma \ref{lem4.4}.
\end{proof}

Now we are in a position to complete the proof of Proposition \ref{main prop}.

\begin{proof}[Proof of Proposition \ref{main prop}]
By multiplying the inequality \eqref{S4eq21} by $k^2 N^{2\eta(k-2)}$ and taking supremum  for  $k\in\N^+$, and then summing up the resulting inequality with $N^{2(\f14-\eta)}\times\eqref{S4eq1},$ we obtain
$$E_0(t) \leq E_0(0)+ C\max\bigl\{N^{\eta-\f14},N^{2(\eta-\f14)} \bigr\} E_1^\f14(t)E_0^\f54(t).$$

Along the same line, by multiplying the inequality \eqref{S4eq31}  by $k^2 N^{2\eta(k-2)}$ and taking supremum  for  $k\in\N^+$, and then summing up the resulting inequality with $N^{2(\f14-\eta)}\times\eqref{S4eq11},$ we achieve
$$E_1(t)
\leq E_1(0)+ C\max\bigl\{N^{\eta-\f14},N^{2(\eta-\f14)} \bigr\} E_0^\f14(t)E_1^\f54(t).$$
This  concludes the proof of Proposition \ref{main prop}.
\end{proof}

\section{Proof of Proposition \ref{propANScyl}}\label{sec5}

In this section, we shall present the proof of Proposition \ref{propANScyl}. First,
let us introduce $u_L$ to be determined by the following anisotropic Stokes system with initial data $u_{\rm in}$ given by \eqref{initialN}:
\begin{equation}\label{5.1}
\left\{
\begin{aligned}
&\p_t u_L-\Delta_\h u_L+\nabla p=0,\quad (t,x)\in\R^+\times \Omega,\\
&\dive u=0,\\
&u_L|_{\p\Om}=0,
\quad u_L|_{t=0} =u_{\rm in}.
\end{aligned}
\right.\end{equation}
It follows from Corollary \ref{corA} in the Appendix \ref{appendixA} that $u_L$ has only frequency $N$ in the
$\theta$ variable.  Precisely, we can write $u_L$ as
\begin{align*}
u_L(t,x)=&
\bigl(\ur_L(t,r,z)\vv e_r+\ut_L(t,r,z)\vv e_\th
+\uz_L(t,r,z)\vv e_z\bigr)\cos N\th\\
&+\bigl(\vr_L(t,r,z)\vv e_r+\vt_L(t,r,z)\vv e_\th
+\vz_L(t,r,z)\vv e_z\bigr)\sin N\th.
\end{align*}
Correspondingly, the pressure function $p$ is of the following form:
$$p(t,x)=
\bar p(t,r,z)\cos N\th+\wt p(t,r,z)\sin N\th.$$
By substituting the above equalities into \eqref{5.1} and comparing the Fourier coefficients, we find
 that $(\ur_L,\vt_L,\uz_L)$ satisfies
\begin{equation}\label{5.2}
\left\{
\begin{split}
&\p_t\ur_L
-\bigl(\pa_r^2+\f{\pa_r}{r}-\frac{1+N^2}{r^2}\bigr)\ur_L
+\f{2N}{r^2}\vt_L+\pa_r\bar p=0,\quad (t,r,z)\in\R^+\times\wt\Omega,\\
&\p_t\vt_L
-\bigl(\pa_r^2+\f{\pa_r}{r}-\frac{1+N^2}{r^2}\bigr)\vt_L
+\f{2N}{r^2}\ur_L-\f Nr\bar p=0,\\
&\p_t\uz_L-\bigl(\pa_r^2+\f{\pa_r}{r}
-\f{N^2}{r^2}\bigr)\uz_L+\pa_z\bar p=0,\\
&\pa_r\ur_L+\f{\ur_L}r+\f{N}r\vt_L+\pa_z\uz_L=0,\\
&(\ur_L,\vt_L,\uz_L)|_{r=1}=0,
\quad(\ur_L,\vt_L,\uz_L)|_{t=0}
=\bigl(N^{\delta}a^r, N^{\delta-1}{b^\th},N^{\delta}a^z\bigr),
\end{split}
\right.
\end{equation}
and $(\vr_L,\ut_L,\vz_L)$ satisfies
\begin{equation}\label{5.3}
\left\{
\begin{split}
&\p_t\vr_L
-\bigl(\pa_r^2+\f{\pa_r}{r}-\frac{1+N^2}{r^2}\bigr)\vr_L
-\f{2N}{r^2}\ut_L+\pa_r\wt p=0,\quad (t,r,z)\in\R^+\times\wt\Omega,\\
&\p_t\ut_L
-\bigl(\pa_r^2+\f{\pa_r}{r}-\frac{1+N^2}{r^2}\bigr)\ut_L
-\f{2N}{r^2}\vr_L+\f Nr\wt p=0,\\
&\p_t\vz_L-\bigl(\pa_r^2+\f{\pa_r}{r}
-\f{N^2}{r^2}\bigr)\vz_L+\pa_z\wt p=0,\\
&\pa_r\vr_L+\f{\vr_L}r-\f{N}r\ut_L+\pa_z\vz_L=0,\\
&(\vr_L,\ut_L,\vz_L)|_{r=1}=0,
\quad(\vr_L,\ut_L,\vz_L)|_{t=0}
=(N^{\delta}b^r,N^{\delta-1}{a^\th},N^{\delta}b^z),
\end{split}
\right.
\end{equation}
where $\wt\Omega\eqdefa \bigl\{(r,z)\in [0,1[\times\R \bigr\}.$

Let us denote $\vv u_L\eqdefa(\ur_L,\vt_L,\uz_L,\vr_L,\ut_L,\vz_L)$ in the rest of this section.

\begin{lem}\label{lem5.1}
{\sl
For any integers $N\geq 3$ and $m\geq0$, we have
\beq\label{5.3a}\|\p_z^m \vv u_L\|_{L^\infty(\R^+;L^2)}^2
+\|\pa_r\p_z^m\vv u_L\|_{L^2(\R^+;L^2)}^2
+N^2\|r^{-1}\p_z^m \vv u_L\|_{L^2(\R^+;L^2)}^2
\lesssim N^{2\delta}\|\p_z^m \vv\al\|_{L^2}^2.\eeq
}\end{lem}

\begin{proof}
Since the systems \eqref{5.2} and \eqref{5.3} are linear, it suffices to prove \eqref{5.3a}  for the case when $m=0$.
Observing that $\vv u_L|_{r=1}=0,$ we get, by using integration by parts and the fourth equations of \eqref{5.2} and \eqref{5.3}, that
\begin{align*}
&\int_{\wt{\Omega}}\bigl(\pa_r\bar p\ur_L-\f Nr\bar p\vt_L+\pa_z\bar p\uz_L\bigr)r\,dr\,d\theta
=-\int_{\wt{\Omega}}\bar p\bigl(\pa_r\ur_L+\f{\ur_L}r+\f{N}r\vt_L+\pa_z\uz_L\bigr)r\,dr\,d\theta=0,\\
&\int_{\wt{\Omega}}\bigl(\pa_r\wt p\vr_L+\f Nr\wt p\ut_L+\pa_z\wt p\vz_L\bigr)r\,dr\,d\theta
=-\int_{\wt{\Omega}}\wt p\bigl(\pa_r\vr_L+\f{\vr_L}r-\f{N}r\ut_L+\pa_z\vz_L\bigr)r\,dr\,d\theta=0,
\end{align*}
 and
 \begin{align*}
-\int_{\wt\Omega} \bigl(\pa_r^2+\f{\pa_r}r\bigr){\vv u}_L
\cdot {\vv u}_L\, rdrdz
=\int_{\wt\Omega}(\pa_r \vv u_L)^2\,rdrdz
=\|\pa_r\vv u_L\|_{L^2}^2.
\end{align*}
 So that we get, by taking $L^2$ inner product of \eqref{5.2} with $(\ur_L,\vt_L,\uz_L)$ and \eqref{5.3} with  $(\vr_L,\ut_L,\vz_L)$, that
\begin{equation}\label{5.4}
\begin{aligned}
\f12\f{d}{dt}\|\vv u_L\|_{L^2}^2
+\|\pa_r\vv u_L\|_{L^2}^2
+N^2\bigl\|\f{\vv u_L}r\bigr\|_{L^2}^2
&\leq4N\Bigl|\bigl(\f{\vt_L}r\big|\f{\ur_L}r\bigr)
-\bigl(\f{\vt_L}r\big|\f{\ur_L}r\bigr)\Bigr|\\
&\leq 2N\bigl\|\f{\vv u_L}r\bigr\|_{L^2}^2,
\end{aligned}
\end{equation}
Notice  that for any $N\geq 3$, there holds $2N\leq \f34N^2$. As a consequence, we deduce from
 \eqref{5.4} that
\begin{align*}
\f12\|\vv u_L\|_{L^\infty(\R^+;L^2)}^2
+\|\pa_r\vv u_L\|_{L^2(\R^+;L^2)}^2
+\f{N^2}4\|r^{-1}\vv u_L\|_{L^2(\R^+;L^2)}^2
\leq\f12\|u_{\rm in}\|_{L^2}^2,
\end{align*}
which together with \eqref{initialN} ensures \eqref{5.3a}. This completes the proof of this lemma.
\end{proof}

Let us now present the proof of Proposition \ref{propANScyl}:

\begin{proof}[Proof of Proposition \ref{propANScyl}] It follows from standard theory of 3-D anisotropic Navier-Stokes system (see
\cite{CDGGbook} for instance) that under the assumptions of Theorem \ref{thm2}, the system $(ANS)$ has a unique solution $u$ on $[0,T^\ast[$
with $T^\ast$ being the maximal time of existence for $u.$
Let us write $u=u_L+w.$  In view of $(ANS)$ and \eqref{5.1}, $w$ satisfies
\begin{equation}\label{5.5}
\left\{
\begin{split}
&\pa_t w-\D_\h w +(u_L+w)\cdot \nabla (u_L+w)
+\nabla q=0,\quad (t,x)\in\R^+\times \Omega,\\
& \dive w=0,\\
& w|_{\pa\Omega}=0 \andf w|_{t=0}=0.
\end{split}
\right.
\end{equation}
Then to prove Proposition \ref{propANScyl}, it suffices to establish  the following estimate for $w$:
\begin{equation}\begin{split}\label{estiwcyl}
\|\pa_3^\ell w\|_{L^\infty_t(L^2)}^2
&+\|\nablah\pa_3^\ell w\|_{L^2_t(L^2)}^2\\
\lesssim& N^{-4(\f14-\delta)}
\sum_{i=0}^{\max\{\ell,1\}}\|\pa_z^i\vv\al\|_{L^2}^4
+N^{-4(1-\delta)}\sum_{j=1}^{\ell+1}
\|\pa_z^j\vv\al\|_{L^2}^4,
\quad\forall\ 0\leq \ell\leq m.
\end{split}\end{equation}

To begin with, we get, by taking $L^2$ inner product of \eqref{5.5} with $w$ and using integration by parts, that
\begin{equation}\begin{split}\label{5.7}
\f12&\|w(t)\|_{L^2}^2
+\|\nablah w\|_{L^2_t(L^2)}^2\\
&=-\int_0^t\int_{\Omega}
\bigl((u_L+w)\cdot\nabla u_L\bigr)\cdot w\,dxdt'\\
&=\int_0^t\int_{\Omega}
\Bigl((u_L^\h\cdot\nablah w\bigr) u_L
-\pa_3(u_L^3 u_L)w
+\nablah\cdot(w^\h\otimes w)u_L
-(w^3\pa_3 u_L) w\Bigr)\,dxdt'\\
&\leq \|u_L\|_{L^4_t(L^4_\h(L^\infty_\v))}
\bigl(\|u_L\|_{{L^4_t(L^4_\h(L^2_\v}))}
\|\nablah w\|_{L^2_t(L^2)}
+\|\pa_3u_L\|_{L^2_t(L^2)}
\|w\|_{L^4_t(L^4_{\rm h}(L^2_{\rm v}))}\bigr)\\
&\quad+\|u_L\|_{L^4_t(L^4_\h(L^\infty_\v))}
\|w\|_{L^4_t(L^4_\h(L^2_\v))}
\|\nablah w\|_{L^2_t(L^2)}
+\|\pa_3u_L\|_{L^2_t(L^2_\h(L^\infty_\v))}
\|w\|_{L^4_t(L^4_\h(L^2_\v))}^2.
\end{split}\end{equation}

It follows from the point-wise estimate \eqref{pointwisecyl} and Lemma \ref{lem5.1} that for any $j\in\N$
\begin{equation}\begin{split}\label{5.8}
\|\pa_3^ju_L\|_{L^2_t(L^2)}
+\|\pa_3^ju_L\|_{L^2_t(L^2_\h(L^\infty_\v))}
&\lesssim\|\pa_3^ju_L\|_{L^2_t(L^2)}
+\|\pa_3^ju_L\|_{L^2_t(L^2)}^{\f12}
\|\pa_3^{j+1}u_L\|_{L^2_t(L^2)}^{\f12}\\
&\lesssim\bigl\|\f{\pa_z^ju_L}r\bigr\|_{L^2_t(L^2)}
+\bigl\|\f{\pa_z^ju_L}r\bigr\|_{L^2_t(L^2)}^{\f12}
\bigl\|\f{\pa_z^{j+1}u_L}r\bigr\|_{L^2_t(L^2)}^{\f12}\\
&\lesssim N^{-(1-\delta)}\|\pa_z^j\vv\al\|_{L^2}^{\f12}\bigl(\|\pa_z^j\vv\al\|_{L^2}^{\f12}
+
\|\pa_z^{j+1}\vv\al\|_{L^2}^{\f12}\bigr).
\end{split}\end{equation}
While by using Minkowski's inequality, \eqref{3.6} and Lemmas \ref{lem5.1}, we infer
\begin{equation}\begin{split}
&\|\pa_3^ju_L\|_{L^4_t(L^4_\h(L^2_\v))}
\leq\|\pa_z^ju_L\|_{L^4_t(L^2_\v(L^4_\h))}\\
&\lesssim\bigl(\|\p_r\pa_z^j\vv u_L\|_{L^2_t(L^2)}^\f14
+\|r^{-1}\pa_z^j\vv u_L\|_{L^2_t(L^2)}^\f14\bigr)
\|r^{-1}\pa_z^j\vv u_L\|_{L^2_t(L^2)}^\f14
\|\pa_z^j\vv u_L\|_{L^\infty_t(L^2)}^\f12\\
&\lesssim N^{-(\f14-\delta)}\|\pa_z^j\vv\al\|_{L^2},
\end{split}\end{equation}
and similarly
\begin{equation}\begin{split}\label{5.10}
\|\pa_3^ju_L\|_{L^4_t(L^4_\h(L^\infty_\v))}
\lesssim &\|\pa_3^ju_L\|_{L^4_t(L^4_\h(L^2_\v))}^{\f12}
\|\pa_3^{j+1} u_L\|_{L^4_t(L^4_\h(L^2_\v))}^{\f12}\\
\lesssim &N^{-(\f14-\delta)}\|\pa_z^j\vv\al\|_{L^2}^{\f12}
\|\pa_z^{j+1}\vv\al\|_{L^2}^{\f12}.
\end{split}\end{equation}
By substituting the estimates \eqref{5.8}-\eqref{5.10} with $j=0$ or $1$ into \eqref{5.7}, we infer
\begin{equation}\begin{split}\label{5.11}
\f12\|w\|_{L^\infty_t(L^2)}^2
+\|&\nablah w\|_{L^2_t(L^2)}^2
\lesssim N^{-2(\f14-\delta)}
\|\vv\al\|_{L^2}^{\f32}
\|\pa_z\vv\al\|_{L^2}^{\f12}
\|\nablah w\|_{L^2_t(L^2)}\\
&+N^{-(\f54-2\delta)}\|\pa_z\vv\al\|_{L^2}
\|\vv\al\|_{L^2}^{\f12}
\|\pa_z\vv\al\|_{L^2}^{\f12}
\|w\|_{L^\infty_t(L^2)}^{\f12}
\|\nablah w\|_{L^2_t(L^2)}^{\f12}\\
&+N^{-(\f14-\delta)}\|\vv\al\|_{L^2}^{\f12}
\|\pa_z\vv\al\|_{L^2}^{\f12}
\|w\|_{L^\infty_t(L^2)}^{\f12}
\|\nablah w\|_{L^2_t(L^2)}^{\f32}\\
&+N^{-(1-\delta)}
\|\pa_z\vv\al\|_{L^2}^{\f12}
\|\pa_z^2\vv\al\|_{L^2}^{\f12}
\|w\|_{L^\infty_t(L^2)}
\|\nablah w\|_{L^2_t(L^2)}.
\end{split}\end{equation}
 We observe that thanks to \eqref{smallANScyl}, the last two terms in the right-hand side of \eqref{5.11} can be absorbed by
  the terms in the left-hand side. Whereas  by using Young's inequality,  the first two terms in the right-hand side of \eqref{5.11} can be bounded
  by
$$\f14\bigl(\|w\|_{L^\infty_t(L^2)}^2
+\|\nablah w\|_{L^2_t(L^2)}^2\bigr)
+N^{-4(\f14-\delta)}
\bigl(\|\vv\al\|_{L^2}^4
+\|\pa_z\vv\al\|_{L^2}^4\bigr)
+N^{-4(1-\delta)}\|\pa_z\vv\al\|_{L^2}^4.$$
As a result, we deduce from \eqref{5.11} that
$$\|w\|_{L^\infty_t(L^2)}^2
+\|\nablah w\|_{L^2_t(L^2)}^2
\lesssim N^{-4(\f14-\delta)}
\bigl(\|\vv\al\|_{L^2}^4
+\|\pa_z\vv\al\|_{L^2}^4\bigr)
+N^{-4(1-\delta)}\|\pa_z\vv\al\|_{L^2}^4,$$
which proves \eqref{estiwcyl} for the case when $\ell=0$.

Next we assume that \eqref{estiwcyl} holds with $\ell=0,\cdots,n-1$ for some $1\leq n\leq m,$
 which together with \eqref{smallANScyl} in particular implies that
\begin{equation}\label{5.12}
\|\pa_3^\ell w\|_{L^\infty_t(L^2)}^2
+\|\nablah\pa_3^\ell w\|_{L^2_t(L^2)}^2
\lesssim \e^4,
\quad\forall\ 0\leq \ell\leq n-1.
\end{equation}
Then by induction, it suffices to prove \eqref{estiwcyl} for $\ell=n$. In order to do so, in view of
\eqref{5.5}, we   write
\begin{equation*}\label{5.13}
\pa_t\pa_3^n w-\D_\h\pa_3^n w
+\sum_{j=0}^n \pa_3^j(u_L+w)
\cdot \nabla\pa_3^{n-j}(u_L+w)
+\nabla\pa_3^nq=0.
\end{equation*}
By taking $L^2$ inner product of the above equation with $\pa_3^n w,$ we find
\begin{equation}\label{5.14}
\f12\f{d}{dt}\|\pa_3^n w\|_{L^2}^2
+\|\nablah\pa_3^n w\|_{L^2}^2
=-\int_{\Om}
\Bigl(\sum_{j=0}^n \pa_3^j(u_L+w)
\cdot \nabla\pa_3^{n-j}(u_L+w)\Bigr)\cdot\pa_3^n w\,dx
\end{equation}
Yet by using integration by parts and the divergence-free condition of $u_L, w$, we obtain
\begin{align*}
&-\sum_{j=0}^n\int_{\Om}
\bigl(\pa_3^j(u_L+w)
\cdot \nabla\pa_3^{n-j}u_L\bigr)\cdot\pa_3^n w\,dx\\
&=\sum_{j=0}^n\int_{\Om}\Bigl(
\bigl(\pa_3^ju_L^\h\cdot\nablah\pa_3^n w\bigr)\cdot\pa_3^{n-j} u_L
-\pa_3\bigl(\pa_3^ju_L^3\pa_3^{n-j} u_L\bigr)
\cdot\pa_3^n w\\
&\qquad\qquad\quad+\nablah\cdot\bigl(\pa_3^jw^\h
\otimes\pa_3^n w\bigr)\cdot\pa_3^{n-j} u_L
-\bigl(\pa_3^jw^3\pa_3^{n-j+1} u_L\bigr)\cdot\pa_3^n w\Bigr)\,dx,
\end{align*}
and
\begin{align*}
&-\sum_{j=0}^n\int_{\Om}
\bigl( \pa_3^j(u_L+w)
\cdot \nabla\pa_3^{n-j}w\bigr)\cdot\pa_3^n w\,dx
=-\sum_{j=1}^n\int_{\Om}
\bigl( \pa_3^j(u_L+w)
\cdot \nabla\pa_3^{n-j}w\bigr)\cdot\pa_3^n w\,dx\\
&=-\sum_{j=1}^n\int_{\Om}
\bigl( \pa_3^j(u_L^\h+w^\h)
\cdot \nablah\pa_3^{n-j}w
-\pa_3^{j-1}\dive_\h(u_L^\h+w^\h)
\cdot\pa_3^{n-j+1}w\bigr)\cdot\pa_3^n w\,dx\\
&=-\sum_{j=1}^n\int_{\Om}
\bigl( \pa_3^j(u_L^\h+w^\h)
\cdot \nablah\pa_3^{n-j}w
-\pa_3^{j-1}\diveh w^\h\cdot
\pa_3^{n-j+1}w\bigr)\cdot\pa_3^n w\,dx\\
&\quad-\sum_{j=0}^{n-1}\int_{\Om}
\pa_3^{j}u_L^\h\cdot\nablah
\bigl(\pa_3^{n-j}w\cdot\pa_3^n w\bigr)\,dx.
\end{align*}
By substituting the above equalities into \eqref{5.14} and then integrating the resulting equality over $[0,t],$ we achieve
\begin{equation}\label{5.15}
\f12\|\pa_3^n w\|_{L^\infty_t(L^2)}^2
+\|\nablah\pa_3^n w\|_{L^2_t(L^2)}^2
\leq\cA_1(t)+\cA_2(t),
\end{equation}
where
\begin{align*}
\cA_1(t)\eqdef&C\sum_{i=0}^{n}
\|\pa_3^iu_L\|_{L^4_t(L^4_\h(L^2_\v))}
\Bigl(\sum_{j=1}^{n}\|\pa_3^j u_L\|_{{L^4_t(L^4_\h(L^2_\v}))}
\|\nablah\pa_3^n w\|_{L^2_t(L^2)}\\
&+\sum_{j=1}^{n+1}\|\pa_3^j u_L\|_{L^2_t(L^2)}
\|\pa_3^n w\|_{L^4_t(L^4_{\rm h}(L^2_{\rm v}))}\\
&+\sum_{j=0}^{n}\bigl(\|\pa_3^j w\|_{L^4_t(L^4_\h(L^2_\v))}
\|\nablah\pa_3^n w\|_{L^2_t(L^2)}
+\|\nablah\pa_3^j w\|_{L^2_t(L^2)}
\|\pa_3^n w\|_{L^4_t(L^4_\h(L^2_\v))}\bigr)\Bigr)\\
&+C\sum_{i=1}^{n+1}
\|\pa_3^iu_L\|_{L^2_t(L^2)}
\sum_{j=0}^{n}\|\pa_3^jw\|_{L^4_t(L^4_\h(L^2_\v))}
\|\pa_3^nw\|_{L^4_t(L^4_\h(L^2_\v))},
\end{align*}
and
\begin{align*}
\cA_2(t)\eqdef C\sum_{j=1}^{n}
\|\pa_3^j w\|_{L^4_t(L^4_\h(L^2_\v))}
\Bigl(&\|\pa_3^{n-j} w\|_{L^4_t(L^4_\h(L^\infty_\v))}
\|\nablah\pa_3^n w\|_{L^2_t(L^2)}\\
&+\|\nablah\pa_3^{n-j} w\|_{L^2_t(L^2_\h(L^\infty_\v))}
\|\pa_3^n w\|_{L^4_t(L^4_\h(L^2_\v))}\Bigr).
\end{align*}
We mention that we  used the inequality $\|\pa_3^jf\|_{L^\infty_\v}\lesssim\|\pa_3^jf\|_{L^2_\v}^{\f12}
\|\pa_3^{j+1}f\|_{L^2_\v}^{\f12}$  for $j\leq n-1$  in the derivation of the estimate of $\cA_1(t)$,

It follows from  \eqref{estiwcyl} with $0\leq\ell\leq n-1$, \eqref{5.8}-\eqref{5.10},~\eqref{5.12} and \eqref{smallANScyl}, that
\begin{align*}
\cA_1(t)\leq&C\Bigl(\sum_{i=0}^{n}
N^{-2(\f14-\delta)}\|\pa_z^i\vv\al\|_{L^2}^2
+\sum_{j=1}^{n+1}N^{-2(1-\delta)}
\|\pa_z^j \vv\al\|_{L^2}^2\\
&\quad+\e^2\sum_{k=0}^{n-1}\bigl(
\|\pa_3^k w\|_{L^\infty_t(L^2)}
+\|\nablah\pa_3^k w\|_{L^2_t(L^2)}\bigr)\Bigr)
\bigl(\|\pa_3^n w\|_{L^\infty_t(L^2)}
+\|\nablah\pa_3^n w\|_{L^2_t(L^2)}\bigr)\\
&+C\Bigl(\sum_{i=0}^{n}
N^{-(\f14-\delta)}\|\pa_z^i\vv\al\|_{L^2}
+\sum_{j=1}^{n+1}N^{-(1-\delta)}
\|\pa_z^j \vv\al\|_{L^2}\Bigr)
\bigl(\|\pa_3^n w\|_{L^\infty_t(L^2)}^2
+\|\nablah\pa_3^n w\|_{L^2_t(L^2)}^2\bigr)\\
\leq&C\Bigl((1+\e^2)\bigl(\sum_{i=0}^{n}
N^{-2(\f14-\delta)}\|\pa_z^i\vv\al\|_{L^2}^2
+\sum_{j=1}^{n+1}N^{-2(1-\delta)}
\|\pa_z^j \vv\al\|_{L^2}^2\bigr)\Bigr)^2\\
&+\bigl(\f16+C\e\bigr)
\bigl(\|\pa_3^n w\|_{L^\infty_t(L^2)}^2
+\|\nablah\pa_3^n w\|_{L^2_t(L^2)}^2\bigr),
\end{align*}
and
\begin{align*}
\cA_2(t)\leq&C\e^2\sum_{k=0}^{n-1}\bigl(
\|\pa_3^k w\|_{L^\infty_t(L^2)}
+\|\nablah\pa_3^k w\|_{L^2_t(L^2)}\bigr)
\bigl(\|\pa_3^n w\|_{L^\infty_t(L^2)}
+\|\nablah\pa_3^n w\|_{L^2_t(L^2)}\bigr)\\
&+C\e^2
\bigl(\|\pa_3^n w\|_{L^\infty_t(L^2)}^2
+\|\nablah\pa_3^n w\|_{L^2_t(L^2)}^2\bigr)
+C\e
\bigl(\|\pa_3^n w\|_{L^\infty_t(L^2)}^{\f52}
+\|\nablah\pa_3^n w\|_{L^2_t(L^2)}^{\f52}\bigr)\\
\leq&C\e^2\Bigl(\sum_{i=0}^{n}
N^{-4(\f14-\delta)}\|\pa_z^i\vv\al\|_{L^2}^4
+\sum_{j=1}^{n+1}N^{-4(1-\delta)}
\|\pa_z^j \vv\al\|_{L^2}^4\Bigr)\\
&+C\e^2
\bigl(\|\pa_3^n w\|_{L^\infty_t(L^2)}^2
+\|\nablah\pa_3^n w\|_{L^2_t(L^2)}^2\bigr)
+C\e
\bigl(\|\pa_3^n w\|_{L^\infty_t(L^2)}^{\f52}
+\|\nablah\pa_3^n w\|_{L^2_t(L^2)}^{\f52}\bigr).
\end{align*}
Here we emphasize that, the term
$C\e
\bigl(\|\pa_3^n w\|_{L^\infty_t(L^2)}^{\f52}
+\|\nablah\pa_3^n w\|_{L^2_t(L^2)}^{\f52}\bigr)$ in the above estimate for $\cA_2(t)$ only appears in the case when $n=1$.

By substituting the above two estimates for $\cA_1(t)$ and $\cA_2(t)$ into \eqref{5.15}, and using the fact that $\e$ is sufficiently small, we achieve
\begin{equation}\begin{split}\label{5.16}
\|\pa_3^n w\|_{L^\infty_t(L^2)}^2
+\|\nablah\pa_3^n& w\|_{L^2_t(L^2)}^2
\leq\bigl(\|\pa_3^n w\|_{L^\infty_t(L^2)}^{\f52}
+\|\nablah\pa_3^n w\|_{L^2_t(L^2)}^{\f52}\bigr)\\
&+C_n\Bigl(\sum_{i=0}^{n}
N^{-4(\f14-\delta)}\|\pa_z^i\vv\al\|_{L^2}^4
+\sum_{j=1}^{n+1}N^{-4(1-\delta)}
\|\pa_z^j \vv\al\|_{L^2}^4\Bigr)
\end{split}\end{equation}
for some positive constant $C_n$ depending only on $n$.

On the other hand, due to $w|_{t=0}=0$,
we denote
\begin{align*}
T^\ast_n\eqdef \sup\Bigl\{t\in ]0, T^\ast[:\,\|\pa_3^n& w\|_{L^\infty_t(L^2)}^2
+\|\nablah\pa_3^n w\|_{L^2_t(L^2)}^2\\
&\leq 2C_n\Bigl(\sum_{i=0}^{n}
N^{-4(\f14-\delta)}\|\pa_z^i\vv\al\|_{L^2}^4
+\sum_{j=1}^{n+1}N^{-4(1-\delta)}
\|\pa_z^j \vv\al\|_{L^2}^4\Bigr)~ \Bigr\}.
\end{align*}
Then for any $t\leq T^\ast_n$, we deduce from  \eqref{smallANScyl} that
$$\|\pa_3^n w\|_{L^\infty_t(L^2)}^2
+\|\nablah\pa_3^n w\|_{L^2_t(L^2)}^2
\leq 2nC_n\e^4,$$
which implies that
$$\|\pa_3^n w\|_{L^\infty_t(L^2)}^{\f12}
+\|\nablah\pa_3^n w\|_{L^2_t(L^2)}^{\f12}
\leq\f13$$
provided $\e$ is sufficiently small.

Therefore we deduce from
\eqref{5.16} that for $t\leq T^\ast_n$
$$\|\pa_3^n w\|_{L^\infty_t(L^2)}^2
+\|\nablah\pa_3^n w\|_{L^2_t(L^2)}^2
\leq\f32C_n\Bigl(\sum_{i=0}^{n}
N^{-4(\f14-\delta)}\|\pa_z^i\vv\al\|_{L^2}^4
+\sum_{j=1}^{n+1}N^{-4(1-\delta)}
\|\pa_z^j \vv\al\|_{L^2}^4\Bigr),$$
which contradicts to the definition of $T_n^\ast$
unless $T_n^\ast=T^\ast=\infty$. Furthermore,for any $t>0$, there holds
$$\|\pa_3^n w\|_{L^\infty_t(L^2)}^2
+\|\nablah\pa_3^n w\|_{L^2_t(L^2)}^2
\leq 2C_n\Bigl(\sum_{i=0}^{n}
N^{-4(\f14-\delta)}\|\pa_z^i\vv\al\|_{L^2}^4
+\sum_{j=1}^{n+1}N^{-4(1-\delta)}
\|\pa_z^j \vv\al\|_{L^2}^4\Bigr),$$
which is  \eqref{estiwcyl} for the case when $\ell=n$. This finishes the proof of \eqref{estiwcyl} by induction.
Then  Proposition \ref{propANScyl} follows from \eqref{5.3a} and \eqref{estiwcyl}.
\end{proof}

\section{Proof of Proposition \ref{propANScylE}}\label{sec6}

This section is devoted to the proof of Proposition \ref{propANScylE}.
Before proceeding, we first list some estimates of $\vv u_0$ and $\vv u_k$  in
terms of the energy functional  $D_j$, which can be proved exactly along the same line to that of Lemma \ref{lembddE}.

\begin{lem}\label{lembddD}
{\sl Let $D_j(t)$ be defined by \eqref{defDj} and
$$\Theta_k\eqdef k^{-\sigma m} N^{-\min\left\{\eta(k-2),
 \left(\f12-\eta-\delta\right)m\right\}},\quad
\wt\Theta_k\eqdef k^{-\sigma(m-1)} N^{-\min\left\{\eta(k-2), \left(\f12-\eta-\delta\right)(m-1)\right\}}.$$
Then for any $k\in\N^+,$ we have
\begin{align*}
&\|\vv u_0\|_{L^4_t(L^4_{\rm h}(L^2_{\rm v}))} \lesssim N^{\eta-\f14}D_0^\f12(t), \quad
\|\p_z \vv u_0\|_{L^4_t(L^4_{\rm h}(L^2_{\rm v}))} \lesssim N^{\eta-\f14}D_1^\f12(t),\\
&\|\vv u_k\|_{L^4_t(L^4_{\rm h}(L^2_{\rm v}))} \lesssim (kN)^{-\f14}\Theta_k D_0^\f12(t),
\quad \|\p_z\vv u_k\|_{L^4_t(L^4_{\rm h}(L^2_{\rm v}))} \lesssim(kN)^{-\f14}\wt\Theta_k D_1^\f12(t),
\end{align*}
and
\begin{align*}
&\|\vv u_0\|_{L^4_t(L^4_{\rm h}(L^\infty_{\rm v}))}+\|\pa_r\vv u_0\|_{L^2_t(L^2_{\rm h}(L^\infty_{\rm v}))}+\bigl\|\f{(\ur_0,\ut_0)}r\bigr\|_{L^2_t(L^2_{\rm h}(L^\infty_{\rm v}))}\lesssim N^{\eta-\f14}D_0^\f14(t)D_1^\f14(t),\\
&\|\vv u_k\|_{L^4_t(L^4_{\rm h}(L^\infty_{\rm v}))} \lesssim (kN)^{-\f14}\Theta_k^{\f12}\wt\Theta_k^{\f12} D_0^\f14(t)D_1^\f14(t),\\
&\|\pa_r \vv u_k\|_{L^2_t(L^2_{\rm h}(L^\infty_{\rm v}))}+kN\bigl\|\f{\vv u_k}r\bigr\|_{L^2_t(L^2_{\rm h}(L^\infty_{\rm v}))}
\lesssim \Theta_k^{\f12}\wt\Theta_k^{\f12} D_0^\f14(t)D_1^\f14(t).
\end{align*}
}\end{lem}

We now present the proof of Proposition \ref{propANScylE}.

\begin{proof}[Proof of Proposition \ref{propANScylE}] Once again, we shall only present the {\it a priori}
estimates for smooth enough solutions of $(ANS).$ In what follows,
we divide the proof into three steps:

\no {\bf Step 1}. The {\it a priori} estimates for $\vv u_0$.

It is easy to observe that all the estimates in Lemmas \ref{lem4.1}-\ref{lem4.4} works for the anisotropic case,
namely with  full Laplacian there being replaced by  the horizontal Laplacian here. Precisely, in view of \eqref{eqtu0} for $\nu=0,$
 we get, by a similar derivation of \eqref{S4eq3}, that
\begin{equation}\begin{split}\label{6.1}
\|\vv u_0&\|_{L^\infty_t(L^2)}^2
+2\|\pa_r \vv u_0\|_{L^2_t(L^2)}^2
+2\bigl\|\f{(\ur_0,\ut_0)}r \bigr\|_{L^2_t(L^2)}^2\\
&\lesssim\sum_{k=1}^\infty
\|\vv u_k\|_{L^4_t(L^4_{\rm h}(L^\oo_{\rm v}))}\|\vv u_k\|_{L^4_t(L^4_{\rm h}(L^2_{\rm v}))} \Bigl(\|\wt\nabla \vv u_0\|_{L^2_t(L^2)}
+\bigl\|\f{(\ur_0,\ut_0)}r\bigr\|_{L^2_t(L^2)}\Bigr).
\end{split}\end{equation}

While by using the divergence-free condition of $u_0,$ $\pa_z u_0^z=-\pa_r\ur_0-\f{\ur_0}r,$ and the point-wise estimate \eqref{pointwisecyl},  for any integer $n\in[0,m-1],$ we find
\begin{equation}\begin{split}\label{6.2}
\|\wt\nabla\pa_z^n& \vv u_0\|_{L^2_t(L^2)}
=\|\pa_r\pa_z^n\vv u_0\|_{L^2_t(L^2)}
+\|\pa_z\pa_z^n(\ur_0,\ut_0)\|_{L^2_t(L^2)}
+\bigl\|\pa_r\pa_z^n\ur_0
+\f{\pa_z^n\ur_0}r\bigr\|_{L^2_t(L^2)}\\
&\lesssim\|\pa_r\pa_z^n\vv u_0\|_{L^2_t(L^2)}
+\bigl\|\f{\pa_z^m (\ur_0,\ut_0)}r\bigr\|
_{L^2_t(L^2)}^{\f1{m-n}}
\bigl\|\f{\pa_z^n(\ur_0,\ut_0)}r\bigr\|
_{L^2_t(L^2)}^{1-\f1{m-n}}
+\bigl\|\f{\pa_z^n\ur_0}r\bigr\|_{L^2_t(L^2)}.
\end{split}\end{equation}
Similarly for any $k\in\N^+,$ in view of the fourth equations of \eqref{eqtuk}
and \eqref{eqtvk},
one has
\begin{equation}\begin{split}\label{6.3}
&\|\wt\nabla\pa_z^n \vv u_k\|_{L^2_t(L^2)}
=\|\pa_r\pa_z^n\vv u_k\|_{L^2_t(L^2)}
+\|\pa_z\pa_z^n(\ur_k,\ut_k,\vr_k,\vt_k)\|_{L^2_t(L^2)}\\
&\quad+\bigl\|\pa_r\pa_z^n\ur_{k}+\f{\pa_z^n\ur_k}r
+kN\f{\pa_z^n\vt_{k}}r\bigr\|_{L^2_t(L^2)}
+\bigl\|\pa_r\pa_z^n\vr_{k}+\f{\pa_z^n\vr_k}r
-kN\f{\pa_z^n\ut_{k}}r\bigr\|_{L^2_t(L^2)}\\
&\lesssim\|\pa_r\pa_z^n\vv u_k\|_{L^2_t(L^2)}
+\bigl\|\f{\pa_z^m\vv u_k}r\bigr\|
_{L^2_t(L^2)}^{\f1{m-n}}
\bigl\|\f{\pa_z^n\vv u_k}r\bigr\|
_{L^2_t(L^2)}^{1-\f1{m-n}}
+kN\bigl\|\f{\pa_z^n\vv u_k}r\bigr\|_{L^2_t(L^2)}.
\end{split}\end{equation}

By substituting the estimate \eqref{6.2} with $n=0$ into \eqref{6.1} and using the interpolation inequality \eqref{S3eq2},
  we infer
  \begin{align*}
&\|\vv u_0\|_{L^\infty_t(L^2)}^2
+2\|\pa_r \vv u_0\|_{L^2_t(L^2)}^2
+2\|\f{(\ur_0,\ut_0)}r \|_{L^2_t(L^2)}^2\\
&\lesssim\sum_{k=1}^\infty
\|\vv u_k\|_{L^4_t(L^4_\h(L^2_\v))}^{\f32}
\|\pa_z\vv u_k\|_{L^4_t(L^4_\h(L^2_\v))}^{\f12}
\Bigl(\|\pa_r \vv u_0\|_{L^2_t(L^2)}\\
&\qquad\qquad\qquad
+\bigl\|\f{\pa_z^m (\ur_0,\ut_0)}r\bigr\|_{L^2_t(L^2)}^{\f1m}
\bigl\|\f{(\ur_0,\ut_0)}r\bigr\|_{L^2_t(L^2)}^{1-\f1m}
+\bigl\|\f{(\ur_0,\ut_0)}r\bigr\|_{L^2_t(L^2)}\Bigr),
\end{align*}
from which, Lemma \ref{lembddD} and the fact that
\beq \label{6.4} \Theta_k\leq k^{-\sigma m}N^\eta,\andf
\wt\Theta_k\leq k^{-\sigma(m-1)}N^\eta,
\quad\forall\ k\in\N^+,\eeq
we deduce that
\begin{align*}
&\|\vv u_0\|_{L^\infty_t(L^2)}^2
+2\|\pa_r \vv u_0\|_{L^2_t(L^2)}^2
+2\bigl\|\f{(\ur_0,\ut_0)}r \bigr\|_{L^2_t(L^2)}^2\\
&\lesssim\sum_{k=1}^{\infty}(kN)^{-\f12}
\Theta_k^{\f32}\wt\Theta_k^{\f12}
D_0^\f34 D_1^\f14
\Bigl(N^{\eta-\f14}D_0^{\f12}
+\bigl\|\f{\pa_z^m (\ur_0,\ut_0)}r\bigr\|_{L^2_t(L^2)}^{\f1m}
N^{\left(1-\f1m\right)\left(\eta-\f14\right)}D_0^{\f12\left(1-\f1{m}\right)}\Bigr)\\
&\lesssim \sum_{k=1}^{\infty}k^{-\sigma\left(2m-\f12\right)-\f12}
N^{3\left(\eta-\f14\right)}D_0^\f54 D_1^\f14\Bigl(1
+\bigl\|\f{\pa_z^m (\ur_0,\ut_0)}r\bigr\|_{L^2_t(L^2)}^{\f1m}
N^{-\f1m\left(\eta-\f14\right)}D_0^{-\f1{2m}}\Bigr)\\
&\lesssim N^{3\left(\eta-\f14\right)}D_0^\f54 D_1^\f14\Bigl(1
+\bigl\|\f{\pa_z^m (\ur_0,\ut_0)}r\bigr\|_{L^2_t(L^2)}^{\f1m}
N^{-\f1m\left(\eta-\f14\right)}D_0^{-\f1{2m}}\Bigr),
\end{align*}
where in the last step, we have used the choice of $\sigma\in\bigl]\f1{2m-3},\f12\bigr[$ so that $\sigma(2m-\f12)+\f12>2\sigma(m-\f32)+\f12>\f32$,
which ensures that the summation $\sum_{k=1}^{\infty}k^{-\sigma\left(2m-\f12\right)-\f12}$ is finite.

As a result, it comes out
\begin{equation}\begin{split}\label{6.5}
N^{2\left(\eta-\f14\right)}\Bigl(\|\vv u_0\|_{L^\infty_t(L^2)}^2
&+2\|\pa_r \vv u_0\|_{L^2_t(L^2)}^2
+2\bigl\|\f{(\ur_0,\ut_0)}r \bigr\|_{L^2_t(L^2)}^2\Bigr)\\
&\lesssim N^{\eta-\f14}D_0^\f54 D_1^\f14\Bigl(1
+\bigl\|\f{\pa_z^m (\ur_0,\ut_0)}r\bigr\|_{L^2_t(L^2)}^{\f1m}
N^{-\f1m\left(\eta-\f14\right)}D_0^{-\f1{2m}}\Bigr).
\end{split}\end{equation}

Whereas in view of \eqref{eqtu0} for $\nu=0,$ we get, by a similar derivation of \eqref{S4eq13}, that
\begin{equation}\label{6.6}
\f12\|\p_z \vv u_0\|_{L^\infty_t(L^2)}^2
+\|\pa_r \p_z\vv u_0\|_{L^2_t(L^2)}^2
+\bigl\|\f{(\p_z\ur_0,\p_z\ut_0)}r
\bigr\|_{L^2_t(L^2)}^2\leq\sum_{j=1}^3I_j(t),
\end{equation}
where $I_j, j=1,2,3,$ is given by \eqref{defI}.

By virtue of  \eqref{defI},  we deduce from ~\eqref{6.2} and Lemma \ref{lembddD} that
\begin{align*}
I_1(t)\lesssim&\|\vv u_0\|_{L^4_t(L^4_{\rm h}(L^\oo_{\rm v}))}\|\p_z\vv u_0\|_{L^4_t(L^4_{\rm h}(L^2_{\rm v}))}\\
&\times\Bigl(\|\pa_r\pa_z\vv u_0\|_{L^2_t(L^2)}
+\bigl\|\f{\pa_z^m (\ur_0,\ut_0)}r\bigr\|
_{L^2_t(L^2)}^{\f1{m-1}}
\bigl\|\f{\pa_z(\ur_0,\ut_0)}r\bigr\|
_{L^2_t(L^2)}^{1-\f1{m-1}}
+\bigl\|\f{\pa_z\ur_0}r\bigr\|_{L^2_t(L^2)}\Bigr)\\
\lesssim& N^{3\left(\eta-\f14\right)}D_0^\f14 D_1^\f54
\Bigl(1+\bigl\|\f{\pa_z^m (\ur_0,\ut_0)}r\bigr\|_{L^2_t(L^2)}^{\f1{m-1}}
N^{-\f1{m-1}\left(\eta-\f14\right)}D_1^{-\f1{2(m-1)}}\Bigr),
\end{align*}
and
\begin{align*}
I_2(t)&\lesssim\|\vv u_0\|_{L^4_t(L^4_{\rm h}(L^\oo_{\rm v}))}\|\p_z\vv u_0\|_{L^4_t(L^4_{\rm h}(L^2_{\rm v}))} \bigl\| \f{(\p_z \ur_0,\p_z \ut_0)}r\bigr\|_{L^2_t(L^2)}\\
&\lesssim N^{3(\eta-\f14)}D_0^\f14 D_1^\f54,
\end{align*}
and
\begin{align*}
I_3(t)\lesssim&\sum_{k=1}^\infty
\|\vv u_k\|_{L^4_t(L^4_\h(L^\infty_\v))}
\|\p_z\vv u_k\|_{L^4_t(L^4_\h(L^2_\v))}\\
&\times\Bigl(\|\pa_r\pa_z\vv u_0\|_{L^2_t(L^2)}
+\bigl\|\f{\pa_z^m (\ur_0,\ut_0)}r\bigr\|_{L^2_t(L^2)}^{\f1{m-1}}
\bigl\|\f{\pa_z(\ur_0,\ut_0)}r\bigr\|
_{L^2_t(L^2)}^{1-\f1{m-1}}
+\bigl\|\f{(\p_z\ur_0,\p_z\ut_0)}r\bigr\|_{L^2_t(L^2)}\Bigr)\\
\lesssim&\sum_{k=1}^{\infty}(kN)^{-\f12}
\Theta_k^{\f12}\wt\Theta_k^{\f32}
D_0^\f14 D_1^\f34
\Bigl(N^{\eta-\f14}D_1^{\f12}
+\bigl\|\f{\pa_z^m (\ur_0,\ut_0)}r\bigr\|_{L^2_t(L^2)}^{\f1{m-1}}
N^{(1-\f1{m-1})(\eta-\f14)}
D_1^{\f12(1-\f1{m-1})}\Bigr)\\
\lesssim& N^{3(\eta-\f14)}D_0^\f14 D_1^\f54
\Bigl(1+\bigl\|\f{\pa_z^m (\ur_0,\ut_0)}r\bigr\|_{L^2_t(L^2)}^{\f1{m-1}}
N^{-\f1{m-1}(\eta-\f14)}D_1^{-\f1{2(m-1)}}\Bigr),
\end{align*}
where we used \eqref{6.4} and again the fact that $\sigma\in\bigl]\f1{2m-3},\f12\bigr[$ so that  the summation $\sum_{k=1}^{\infty}k^{-\sigma\left(2m-\f32\right)-\f12}$ is finite.

By substituting the above three estimates into \eqref{6.6}, we achieve
\begin{equation}\begin{split}\label{6.7}
N^{2\left(\eta-\f14\right)}\Bigl(\|\p_z \vv u_0&\|_{L^\infty_t(L^2)}^2
+2\|\pa_r \p_z\vv u_0\|_{L^2_t(L^2)}^2
+2\bigl\|\f{(\p_z\ur_0,\p_z\ut_0)}r
\bigr\|_{L^2_t(L^2)}^2\Bigr)\\
&\lesssim N^{\eta-\f14}D_0^\f14 D_1^\f54
\Bigl(1+\bigl\|\f{\pa_z^m (\ur_0,\ut_0)}r\bigr\|_{L^2_t(L^2)}^{\f1{m-1}}
N^{-\f1{m-1}\left(\eta-\f14\right)}D_1^{-\f1{2(m-1)}}\Bigr).
\end{split}\end{equation}

\no{\bf Step 2.} The {\it a priori} estimates for $\vv u_k$.

\begin{itemize}
\item The energy estimate of $\vv u_k.$
\end{itemize}

In view of \eqref{eqtuk} and \eqref{eqtvk} for $\nu=0,$ we get, by a similar derivation of
 \eqref{S4eq23} that
\begin{equation}\label{6.8}
\|\vv u_k\|_{L^\infty_t(L^2)}^2
+2\|\pa_r \vv u_k\|_{L^2_t(L^2)}^2
+\f{k^2N^2}2\|\f{\vv u_k}r \|_{L^2_t(L^2)}^2
\leq \|\vv u_k(0)\|_{L^2}^2
+CII_1(t)+CII_2(t),
\end{equation}
where $II_j, j=1,2,$ is given by \eqref{defII}.

 Thanks to \eqref{defII}, we get, by using \eqref{6.2} and Lemma \ref{6.1}, that
\begin{equation}\begin{split}\label{6.9}
II_1(t)&\eqdef\|\vv u_k\|_{L^4_t(L^4_\h(L^2_\v))}\|\vv u_k\|_{L^4_t(L^4_\h(L^\infty_\v))}
\Bigl(\|\wt\nabla \vv u_0\|
_{L^2_t(L^2)}
+\bigl\|\f{(\ur_0,\ut_0)}r\bigr\|
_{L^2_t(L^2)}\Bigr),\\
&\lesssim(kN)^{-\f12}\Theta_k^{\f32}
\wt\Theta_k^{\f12}D_0^\f54 D_1^\f14 N^{\eta-\f14}\Bigl(1+\bigl\|\f{\pa_z^m (\ur_0,\ut_0)}r\bigr\|_{L^2_t(L^2)}^{\f1m}
N^{-\f1m(\eta-\f14)}D_0^{-\f1{2m}}\Bigr)\\
&\lesssim k^{-\f14}\Theta_k^{2}
D_0^\f54 D_1^\f14 N^{\f12(\eta-\delta-1)}\Bigl(1+\bigl\|\f{\pa_z^m (\ur_0,\ut_0)}r\bigr\|_{L^2_t(L^2)}^{\f1m}
N^{-\f1m(\eta-\f14)}D_0^{-\f1{2m}}\Bigr).
\end{split}\end{equation}
where in the last step we have used the facts: $\sigma\in\bigl]\f1{2m-3},\f12\bigr[$ and $\eta\in \left[0,\f12-\delta\right[,$ so
that
 $$k^{-\f12}N^{-\f12+\eta+\delta}\wt\Theta_k
\leq k^{-\sigma}N^{-\f12+\eta+\delta}\wt\Theta_k
\leq\Theta_k.$$

 The estimate of $II_2(t)$ is more sophisticated. We first observe from \eqref{3.7} and Lemma \ref{6.1} that
 \beq\label{6.9a}
 \begin{split}
 \|\vv u_{k_1}\|_{L^{\f2{1-2\ve}}_t(L^{\f1\ve}_\h(L^2_\v))}\lesssim&\|\vv u_{k_1}\|_{L^\infty_t(L^2)}^{2\ve}
\bigl(\|\p_r\vv u_{k_1}\|_{L^2_t(L^2)}^{\f12-\ve}+\|\f{\vv u_{k_1}}{r}\|_{L^2_t(L^2)}^{\f12-\ve}\bigr)\|\f{\vv u_{k_1}}{r}\|_{L^2_t(L^2)}^{\f12-\ve}\\
\lesssim &(k_1N)^{-\left(\f12-\ve\right)}\Theta_{k_1}D_0^{\f12},\andf\\
\|\vv u_{k_2}\|_{L^{\f1\ve}_t(L^{\f2{1-2\ve}}_\h(L^2_\v))}\lesssim&\|\vv u_{k_2}\|_{L^\infty_t(L^2)}^{1-2\ve}
\bigl(\|\p_r\vv u_{k_2}\|_{L^2_t(L^2)}^{\ve}+\|\f{\vv u_{k_2}}{r}\|_{L^2_t(L^2)}^{\ve}\bigr)\|\f{\vv u_{k_2}}{r}\|_{L^2_t(L^2)}^{\ve}\\
\lesssim &(k_2N)^{-\ve}\Theta_{k_2}D_0^{\f12}.
\end{split} \eeq
Similarly, we get, by using \eqref{S3eq2}, that
\beq\label{6.9b}
\begin{split}
\|\p_z\vv u_{k_2}\|_{L^{\f1\ve}_t(L^{\f2{1-2\ve}}_\h(L^2_\v))}
\lesssim& (k_2N)^{-\ve}\wt\Theta_{k_2}D_1^{\f12},\\
\|\vv u_{k_2}\|_{L^{\f1\ve}_t(L^{\f2{1-2\ve}}_\h(L^\infty_\v))}
 \lesssim & (k_2N)^{-\ve}\Theta_{k_2}^{\frac12}\wt\Theta_{k_2}^{\frac12} D_0^{\f14}D_1^{\f14}.
 \end{split}
\eeq

Notice that due to symmetry of the indices, $k_1,k_2,$ in \eqref{4.16}, we only consider the part of  summation where $k_1\geq k_2.$
 As a consequence, we get, by using \eqref{4.16}, ~\eqref{6.3} and Lemma \ref{6.1}, that for $\varepsilon\eqdef\f12-\sigma>0$
\begin{align*}
II_2(t)&\lesssim\sum_{\substack{
k_1+k_2=k,~k_1\geq k_2\\
\text{ or }k_1-k_2=k}}
\|\vv u_{k_1}\|_{L^{\f2{1-2\ve}}_t(L^{\f1\ve}_\h(L^2_\v))}
\|\vv u_{k_2}\|_{L^{\f1\ve}_t(L^{\f2{1-2\ve}}_\h(L^\infty_\v))} \bigl(\|\wt\nabla \vv u_k\|_{L^2_t(L^2)}
+kN\bigl\|\f{\vv u_k}r\bigr\|_{L^2_t(L^2)}\bigr)\\
&\lesssim D_0^\f54 D_1^\f14\Theta_k
\Bigl(1+\Theta^{-\f1m}_k
\bigl\|\f{\pa_z^m\vv u_k}r\bigr\|
_{L^2_t(L^2)}^{\f1m}D_0^{-\f1{2m}}\Bigr)
\sum_{\substack{
k_1+k_2=k,~k_1\geq k_2\\
\text{ or }k_1-k_2=k}}
k_1^{-\f12+\ve}k_2^{-\ve}N^{-\f12}
\Theta_{k_1}\Theta_{k_2}^{\f12}\wt\Theta_{k_2}^{\f12}.
\end{align*}

Let us first consider the case when $k_1\leq\bigl[\f{(\f12-\eta-\delta)m}\eta\bigr]+2$. In this case, $\Theta_{k_1}=k_1^{-\sigma m} N^{-\eta(k_1-2)}$
and $\Theta_{k_2}=k_2^{-\sigma m} N^{-\eta(k_2-2)}.$ Observing that both $k_1+k_2=k$ and $k_1\geq k_2$ implies $k_1\geq k/2$, so that we have
\begin{align*}
\sum_{\substack{k_1+k_2=k\\k_1\geq k_2}}
k_1^{-\f12+\ve}k_2^{-\ve}N^{-\f12}
\Theta_{k_1}\Theta_{k_2}^{\f12}\wt\Theta_{k_2}^{\f12}
&\lesssim N^{-\f12}\sum_{\substack{k_1+k_2=k\\k_1\geq k_2}}
k_1^{-\sigma m-\f12+\ve}
k_2^{-\sigma\left(m-\f12\right)-\ve}N^{\cB}\\
&\lesssim k^{-\sigma m-\f12+\ve}N^{-\f12}
\sum_{\substack{k_1+k_2=k\\k_1\geq k_2}}
k_2^{-\sigma\left(m-\f12\right)-\ve}N^{\cB},
\end{align*}
where
$$\cB\eqdef-\eta(k_1-2)-\f\eta2(k_2-2)
-\f12\min\bigl\{\eta(k_2-2), \bigl(\f12-\eta-\delta\bigr)(m-1)\bigr\}.$$
It is easy to observe that when $k_2\leq\bigl[\f{(\f12-\eta-\delta)(m-1)}\eta\bigr]+2$, there holds
$$\cB=-\eta(k_1-2)-\eta(k_2-2)=2\eta-\eta(k-2).$$
While when $k_2>\bigl[\f{(\f12-\eta-\delta)(m-1)}\eta\bigr]+2$, recalling that $k_1\geq k_2$ and $m\geq3$, we have
\begin{align*}
\cB&=-\eta(k_1-2)-\f\eta2(k_2-2)
-\f12\bigl(\f12-\eta-\delta\bigr)(m-1)\\
&=2\eta-\f\eta2k_1-\f\eta2(k-2)
-\f12\bigl(\f12-\eta-\delta\bigr)(m-1)\\
&<2\eta-\f\eta2
\Bigl(\f{(\f12-\eta-\delta)(m-1)}\eta+2\Bigr)
-\f\eta2(k-2)-\f12\bigl(\f12-\eta-\delta\bigr)(m-1)\\
&=2\eta-\f\eta2(k-2)-\f{m}2\bigl(\f12-\eta-\delta\bigr)
-\f12{(\f12-\eta-\delta)(m-2)}-\eta\\
&<2\eta-\min\bigl\{\eta(k-2), m\bigl(\f12-\eta-\delta\bigr)\bigr\}.
\end{align*}
As a consequence, we  deduce that
\begin{align*}
&\sum_{\substack{k_1+k_2=k\\~k_1\geq k_2}}
k_1^{-\f12+\ve}k_2^{-\ve}N^{-\f12}
\Theta_{k_1}\Theta_{k_2}^{\f12}\wt\Theta_{k_2}^{\f12}\\
&\lesssim k^{-\f12+\ve}N^{2(\eta-\f14)}
k^{-\sigma m}
N^{-\min\left\{\eta(k-2), m\left(\f12-\eta-\delta\right)\right\}}\sum_{k_2\in\N^+}
k_2^{-\sigma\left(m-\f12\right)-\ve}\\
&\lesssim k^{-\f12+\ve}N^{2(\eta-\f14)}
\Theta_k,
\end{align*}
where in the last step we  used the choice of $\sigma\in\bigl]\f1{2m-3},\f12\bigr[$ and $\ve=\f12-\sigma$ so that
\begin{equation}\label{6.10}
\sigma\bigl(m-\f12\bigr)+\ve=\sigma\bigl(m-\f32\bigr)+\f12>1.
\end{equation}

On the other hand, when $k_1>\bigl[\f{(\f12-\eta-\delta)m}\eta\bigr]+2$, we have $\Theta_{k_1}=k_1^{-\sigma m} N^{-\left(\f12-\eta-\delta\right)m}$, and thus
\begin{align*}
&\sum_{\substack{k_1+k_2=k\\k_1\geq k_2}}
k_1^{-\f12+\ve}k_2^{-\ve}N^{-\f12}
\Theta_{k_1}\Theta_{k_2}^{\f12}\wt\Theta_{k_2}^{\f12}\\
&\lesssim N^{-\f12}N^{-\left(\f12-\eta-\delta\right)m}
k^{-\sigma m-\f12+\ve}
\sum_{k_2\in\N^+}
k_2^{-\sigma(m-\f12)-\ve}
N^{-\min\bigl\{\eta(k_2-2), \left(\f12-\eta-\delta\right)(m-1)\bigr\}}\\
&\lesssim k^{-\f12+\ve}N^{\eta-\f12}\Theta_k.
\end{align*}

When $k_1-k_2=k$, one has $k_1\geq k+1$, and thus
$$k_1^{-\f12+\ve}\Theta_{k_1}
<k^{-\f12+\ve}\Theta_{k},$$
from which and \eqref{6.10}, we deduce that
\begin{align*}
&\sum_{k_1-k_2=k}k_1^{-\f12+\ve}k_2^{-\ve}N^{-\f12}
\Theta_{k_1}\Theta_{k_2}^{\f12}\wt\Theta_{k_2}^{\f12}\\
&\leq N^{-\f12}k^{-\f12+\ve}\Theta_{k}
\sum_{k_2\in\N^+}k_2^{-\sigma(m-\f12)-\ve}
N^{-\min\bigl\{\eta(k_2-2),\left(\f12-\eta-\delta\right)(m-1)\bigr\}}\\
&\lesssim k^{-\f12+\ve}N^{\eta-\f12}\Theta_k.
\end{align*}

By summarizing the above
 above estimates, we conclude that
\begin{equation}\label{6.11}
II_2(t)\lesssim k^{-\f12+\ve}N^{2(\eta-\f14)}
D_0^\f54 D_1^\f14\Bigl(1+\Theta^{-\f1m}_k
\bigl\|\f{\pa_z^m\vv u_k}r\bigr\|
_{L^2_t(L^2)}^{\f1m}D_0^{-\f1{2m}}\Bigr)
\Theta_k^2.
\end{equation}

By substituting the estimates \eqref{6.9} and \eqref{6.11} into \eqref{6.8}, and recalling that $\varepsilon=\f12-\sigma$, we achieve
\begin{equation}\begin{split}\label{6.12}
\Theta_k^{-2}&\Bigl(\|\vv u_k\|_{L^\infty_t(L^2)}^2
+2\|\pa_r \vv u_k\|_{L^2_t(L^2)}^2
+\f{k^2N^2}2\|\f{\vv u_k}r \|_{L^2_t(L^2)}^2\Bigr)\\
\lesssim  &\Theta_k^{-2}\|\vv u_k(0)\|_{L^2}^2+k^{-\sigma}N^{2(\eta-\f14)}
D_0^\f54 D_1^\f14\Bigl(1+\Theta^{-\f1m}_k
\bigl\|\f{\pa_z^m\vv u_k}r\bigr\|
_{L^2_t(L^2)}^{\f1m}D_0^{-\f1{2m}}\Bigr)\\
&+ k^{-\f14}D_0^\f54 D_1^\f14 N^{\f12(\eta-\delta-1)}\Bigl(1+\bigl\|\f{\pa_z^m (\ur_0,\ut_0)}r\bigr\|_{L^2_t(L^2)}^{\f1m}
N^{-\f1m(\eta-\f14)}D_0^{-\f1{2m}}\Bigr).
\end{split}\end{equation}

\begin{itemize}
\item The energy estimate of $\p_z\vv u_k.$
\end{itemize}

In view of  \eqref{eqtuk}
and \eqref{eqtvk} for $\nu=0,$
we get, by a similar derivation of \eqref{4.22}, that
\begin{equation}\label{6.13}
\|\p_z\vv u_k\|_{L^\infty_t(L^2)}^2
+2\|\pa_r \p_z\vv u_k\|_{L^2_t(L^2)}^2
+\f{k^2N^2}2\|\f{\p_z\vv u_k}r \|_{L^2_t(L^2)}^2
\leq \|\pa_z\vv u_k(0)\|_{L^2}^2
+C\sum_{j=1}^3 \wt{III}_j(t),
\end{equation}
where $\wt{III}_j, j=1,2,3,$ is similar to the $III_j$'s in \eqref{4.22}, yet which needs to be handled more subtly. Indeed we first get, by using Lemma \ref{lembddD} and the fact that $\pa_zu_0^z=-\pa_ru^r_0-\f{u^r_0}r$, that
\begin{equation}\begin{split}\label{6.14}
\wt{III}_1(t)&\eqdef\|\p_z\vv u_k\|_{L^4_t(L^4_\h(L^2_\v))}
\Bigl(kN\|\p_z\vv u_0\|_{L^4_t(L^4_\h(L^2_\v))}
\bigl\|\f{ \vv u_k}r\bigr\|_{L^2_t(L^2_{\rm h}(L^\infty_{\rm v}))}\\
&\quad+\|\p_z u_0^r\|_{L^4_t(L^4_\h(L^2_\v))}
\|\pa_r \vv u_k\|_{L^2_t(L^2_\h(L^\infty_\v))}
+\|\p_z u_0^z\|_{L^2_t(L^2_\h(L^\infty_\v))}
\|\pa_z \vv u_k\|_{L^4_t(L^4_\h(L^2_\v))}\Bigr)\\
&\lesssim(kN)^{-\f14}\wt\Theta_k D_1^{\f12} \Bigl(N^{\eta-\f14}\Theta_k^{\f12}\wt\Theta_k^{\f12} D_0^{\f14}D_1^{\f34}
+N^{\eta-\f14}(kN)^{-\f14}\wt\Theta_k D_0^{\f14}D_1^{\f34}\Bigr)\\
&\lesssim N^{\eta-\f12}k^{-\f14-\f12\sigma}
D_0^{\f14}D_1^{\f54}\wt\Theta_k^2,
\end{split}\end{equation}
where we have used the fact that $\Theta_k\leq k^{-\sigma}\wt\Theta_k$ in the last step.

Whereas it follows from \eqref{6.2} that
\begin{equation}\begin{split}\label{6.15}
\wt{III}_2(t)&\eqdef\|\p_z\vv u_k\|_{L^4_t(L^4_\h(L^2_\v))}
\|\vv u_k\|_{L^4_t(L^4_\h(L^\infty_\v))}
\bigl(\|\wt\nabla \p_z\vv u_0\|_{L^2_t(L^2)}
+\bigl\|\f{(\p_z\ur_0,\p_z\ut_0)}r
\bigr\|_{L^2_t(L^2)}\bigr)\\
&\quad+\|\p_z\vv u_k\|_{L^4_t(L^4_\h(L^2_\v))} \Bigl(\|\p_z\vv u_k\|_{L^4_t(L^4_\h(L^2_\v))}
\bigl\|\f{(\ur_0,\ut_0)}r\bigr\|
_{L^2_t(L^2_\h(L^\infty_\v))}\\
&\qquad\qquad\qquad\qquad\qquad
+\|\p_z(u_k^r,v^r_k)\|_{L^4_t(L^4_\h(L^2_\v))}
\|\pa_r\vv u_0\|_{L^2_t(L^2_\h(L^\infty_\v))}\\
&\qquad\qquad\qquad\qquad\qquad
+\|\p_z(u_k^z,v^z_k)\|_{L^2_t(L^2_\h(L^\infty_\v))}
\|\pa_z\vv u_0\|_{L^4_t(L^4_\h(L^2_\v))}\Bigr)\\
&\lesssim N^{\eta-\f14}(kN)^{-\f12}D_0^\f14 D_1^\f54 \Bigl(1+\bigl\|\f{\pa_z^m (\ur_0,\ut_0)}r\bigr\|_{L^2_t(L^2)}^{\f1{m-1}}
N^{-\f1{m-1}(\eta-\f14)}D_1^{-\f1{2(m-1)}}\Bigr)
\wt\Theta^2_k\\
&\quad+ N^{\eta-\f12}k^{-\f14-\f12\sigma}
D_0^{\f14}D_1^{\f54}\wt\Theta_k^2.
\end{split}\end{equation}

Finally in view of \eqref{4.24}, we deduce from Lemma \ref{lembddD}, \eqref{6.9a}  and \eqref{6.9b} that
\begin{align*}
&\wt{III}_3(t)\eqdef\int_0^t \Bigl|\bigl(\pa_z(F_k,G_k)
\big|\pa_z\vv u_k\bigr)\Bigr|\,dt'\\
&\lesssim\bigl(\|\wt\nabla \p_z \vv u_k\|_{L^2_t(L^2)}+kN\bigl\|\f{\p_z\vv u_k}r\bigr\|_{L^2_t(L^2)}\bigr)\\
&\qquad\times\Bigl(\sum_{\substack{
k_1+k_2=k,~k_1\geq k_2\\
\text{ or }k_1-k_2=k}}
\|\p_z \vv u_{k_1}\|
_{L^{\f2{1-2\ve}}_t(L^{\f1\ve}_\h(L^2_\v))}
\|\vv u_{k_2}\|
_{L^{\frac1\ve}_t(L^{\f2{1-2\ve}}_\h(L^\infty_\v))}\\
&\qquad\qquad+\sum_{\substack{
k_1+k_2=k,~k_1< k_2\\
\text{ or }k_2-k_1=k}}
\|\p_z \vv u_{k_1}\|
_{L^{\frac1\ve}_t(L^{\f2{1-2\ve}}_\h(L^2_\v))}
\|\vv u_{k_2}\|
_{L^{\f2{1-2\ve}}_t(L^{\f1\ve}_\h(L^\infty_\v))}\Bigr)\\
&\lesssim D_0^\f14 D_1^\f54\wt\Theta_k
\Bigl(1+\wt\Theta^{-\f1{m-1}}_k
\bigl\|\f{\pa_z^m\vv u_k}r\bigr\|
_{L^2_t(L^2)}^{\f1{m-1}}D_1^{-\f1{2(m-1)}}\Bigr)\cG,
\end{align*}
where
$$\cG\eqdef\sum_{\substack{
k_1+k_2=k,~k_1\geq k_2\\
\text{ or }k_1-k_2=k}}
k_1^{-\f12+\ve}k_2^{-\ve}N^{-\f12}
\wt\Theta_{k_1}\Theta_{k_2}^{\f12}\wt\Theta_{k_2}^{\f12}
+\sum_{\substack{
k_1+k_2=k,~k_1< k_2\\
\text{ or }k_2-k_1=k}}
k_1^{-\ve}k_2^{-\f12+\ve}N^{-\f12}
\wt\Theta_{k_1}
\Theta_{k_2}^{\f12}\wt\Theta_{k_2}^{\f12}.$$
Observing that $\Theta_k\leq k^{-\sigma}\wt\Theta_k$ again, we infer
\begin{align*}
\cG
&\leq\sum_{\substack{
k_1+k_2=k,~k_1\geq k_2\\
\text{ or }k_1-k_2=k}}
k_1^{-\f12+\ve}k_2^{-\ve-\f\sigma2}N^{-\f12}
\wt\Theta_{k_1}\wt\Theta_{k_2}
+\sum_{\substack{
k_1+k_2=k,~k_1< k_2\\
\text{ or }k_2-k_1=k}}
(k_1^{-\ve}k_2^{-\f\sigma2})k_2^{-\f12+\ve}N^{-\f12}
\wt\Theta_{k_1}\wt\Theta_{k_2}\\
&\leq2\sum_{\substack{
k_1+k_2=k,~k_1\geq k_2\\
\text{ or }k_1-k_2=k}}
k_1^{-\f12+\ve}k_2^{-\ve-\f\sigma2}N^{-\f12}
\wt\Theta_{k_1}\wt\Theta_{k_2}.
\end{align*}
Noticing that $k_1+k_2=k$ and $k_1\geq k_2$ implies $k_1\geq k/2,$ and if $k_1\leq\bigl[\f{(\f12-\eta-\delta)(m-1)}\eta\bigr]+2$, we have $\wt\Theta_{k_1}=k_1^{-\sigma(m-1)} N^{-\eta(k_1-2)}$ and $\wt\Theta_{k_2}=k_2^{-\sigma(m-1)} N^{-\eta(k_2-2)}$, so that there holds
\begin{align*}
\sum_{\substack{k_1+k_2=k\\~k_1\geq k_2}}
k_1^{-\f12+\ve}k_2^{-\ve-\f\sigma2}N^{-\f12}
\wt\Theta_{k_1}\wt\Theta_{k_2}
&\lesssim N^{-\f12}N^{\eta(2-k)+2\eta}
k^{-\sigma(m-1)-\f12+\ve}\sum_{k_2\in\N^+}
k_2^{-\sigma(m-\f12)-\ve}\\
&\lesssim k^{-\f12+\ve}N^{2(\eta-\f14)}\wt\Theta_k,
\end{align*}
where in the last step we have used \eqref{6.10} to guarantee the summation for $k_2$ is finite.

While for the case when $k_1>\bigl[\f{(\f12-\eta-\delta)(m-1)}\eta\bigr]+2$, we have $\wt\Theta_{k_1}=k_1^{-\sigma(m-1)} N^{-\left(\f12-\eta-\delta\right)(m-1)}$, and thus
\begin{align*}
&\sum_{\substack{k_1+k_2=k\\~k_1\geq k_2}}
k_1^{-\f12+\ve}k_2^{-\ve-\f\sigma2}N^{-\f12}
\wt\Theta_{k_1}\wt\Theta_{k_2}\\
&\lesssim N^{-\f12}N^{-\left(\f12-\eta-\delta\right)(m-1)}
k^{-\sigma(m-1)-\f12+\ve}
\sum_{k_2\in\N^+}
k_2^{-\sigma(m-\f12)-\ve}
N^{-\min\bigl\{\eta(k_2-2),\left(\f12-\eta-\delta\right)(m-1)\bigr\}}\\
&\lesssim k^{-\f12+\ve}N^{\eta-\f12}\wt\Theta_k.
\end{align*}

On the other hand,  when $k_1-k_2=k$, one has $k_1\geq k+1$, so that
$$k_1^{-\f12+\ve}\wt\Theta_{k_1}
<k^{-\f12+\ve}\wt\Theta_{k}.$$
As a result, we deduce that
\begin{align*}
&\sum_{k_1-k_2=k}
k_1^{-\f12+\ve}k_2^{-\ve-\f\sigma2}N^{-\f12}
\wt\Theta_{k_1}\wt\Theta_{k_2}\\
&\leq N^{-\f12}k^{-\f12+\ve}\wt\Theta_{k}
\sum_{k_2\in\N^+}
k_2^{-\sigma(m-\f12)-\ve}
N^{-\min\{\eta(k_2-2),(\f12-\eta-\delta)(m-1)\}}\\
&\lesssim k^{-\f12+\ve}N^{\eta-\f12}\wt\Theta_k.
\end{align*}

By summarizing the above  estimates, we arrive at
$$\cG\lesssim k^{-\f12+\ve}N^{2(\eta-\f14)}\wt\Theta_k,$$
and thus
\begin{equation}\begin{split}\label{6.16}
\wt{III}_3(t)&\lesssim N^{2(\eta-\f14)}k^{-\f12+\ve}D_0^\f14 D_1^\f54
\Bigl(1+\wt\Theta^{-\f1{m-1}}_k
\bigl\|\f{\pa_z^m\vv u_k}r\bigr\|
_{L^2_t(L^2)}^{\f1{m-1}}D_1^{-\f1{2(m-1)}}\Bigr)
\wt\Theta_k^2.
\end{split}\end{equation}

By substituting the estimates \eqref{6.14}-\eqref{6.16} into \eqref{6.13}, and recalling that $\varepsilon=\f12-\sigma$ and $\sigma<\f12$ so that $N^{\eta-\f12}k^{-\f14-\f\sigma2}
<N^{2(\eta-\f14)}k^{-\sigma}$, we achieve
\begin{equation}\begin{split}\label{6.17}
\wt\Theta_k^{-2}&\Bigl(\|\pa_z\vv u_k\|_{L^2}^2
+2\|\pa_r \p_z\vv u_k\|_{L^2_t(L^2)}^2
+\f{k^2N^2}2\|\f{\p_z\vv u_k}r \|_{L^2_t(L^2)}^2\Bigr)
\\
\lesssim &\wt\Theta_k^{-2}\|\pa_z\vv u_k(0)\|_{L^2}^2+N^{2(\eta-\f14)}k^{-\sigma}
D_0^\f14 D_1^\f54
\Bigl(1+\wt\Theta^{-\f1{m-1}}_k
\bigl\|\f{\pa_z^m\vv u_k}r\bigr\|
_{L^2_t(L^2)}^{\f1{m-1}}D_1^{-\f1{2(m-1)}}\Bigr)\\
&+N^{\eta-\f14}(kN)^{-\f12}
D_0^\f14 D_1^\f54
\Bigl(1+\bigl\|\f{\pa_z^m (\ur_0,\ut_0)}r\bigr\|_{L^2_t(L^2)}^{\f1{m-1}}
N^{-\f1{m-1}(\eta-\f14)}D_1^{-\f1{2(m-1)}}\Bigr).
\end{split}\end{equation}

\no{\bf Step 3.} Conclusion of the proof.

By taking supremum of $k\in\N^+$ for \eqref{6.12} and then summing up the resulting inequality with \eqref{6.5}, we find
\begin{align*}
D_0(t)\leq & D_0(0)+N^{\eta-\f14}D_0^\f54(t) D_1^\f14(t)\Bigl(1
+\bigl\|\f{\pa_z^m (\ur_0,\ut_0)}r\bigr\|_{L^2_t(L^2)}^{\f1m}
N^{-\f1m(\eta-\f14)}D_0^{-\f1{2m}}(t)\Bigr)\\
&+N^{\f12(\eta-\delta-1)}D_0^\f54(t) D_1^\f14(t) \Bigl(1+\bigl\|\f{\pa_z^m (\ur_0,\ut_0)}r\bigr\|_{L^2_t(L^2)}^{\f1m}
N^{-\f1m(\eta-\f14)}D_0^{-\f1{2m}}(t)\Bigr)\\
&+\sup_{k\in\N^+}k^{-\sigma}N^{2(\eta-\f14)}
D_0^\f54(t) D_1^\f14(t)\Bigl(1+\Theta^{-\f1m}_k
\bigl\|\f{\pa_z^m\vv u_k}r\bigr\|
_{L^2_t(L^2)}^{\f1m}D_0^{-\f1{2m}}(t)\Bigr),
\end{align*}
which together with the facts: $\Theta_k^{-\f1m}= k^{\sigma} N^{\min\{\f{\eta(k-2)}m,\f12-\eta-\delta\}}
\leq k^{\sigma} N^{\f12-\eta-\delta}$ and  $N^{\f12(\eta-\delta-1)}<N^{\eta-\f14},$ ensures  \eqref{ANScylE1}.

Along the same line, we get, by taking supremum of $k\in\N^+$ for  \eqref{6.17} and then summing up the resulting inequality with \eqref{6.7}, that
\begin{align*}
D_1(t) \leq & D_1(0)+
N^{\eta-\f14}D_0^\f14(t) D_1^\f54(t)
\Bigl(1+\bigl\|\f{\pa_z^m (\ur_0,\ut_0)}r\bigr\|_{L^2_t(L^2)}^{\f1{m-1}}
N^{-\f1{m-1}(\eta-\f14)}D_1^{-\f1{2(m-1)}}(t)\Bigr)\\
&+\sup_{k\in\N^+}k^{-\sigma}N^{2(\eta-\f14)}
D_0^\f14(t) D_1^\f54(t)
\Bigl(1+\wt\Theta^{-\f1{m-1}}_k
\bigl\|\f{\pa_z^m\vv u_k}r\bigr\|
_{L^2_t(L^2)}^{\f1{m-1}}D_1^{-\f1{2(m-1)}}(t)\Bigr).
\end{align*}
which together with the fact: $\wt\Theta_k^{-\f1{m-1}}= k^{\sigma} N^{\min\{\f{\eta(k-2)}{m-1},\f12-\eta-\delta\}}
\leq k^{\sigma} N^{\f12-\eta-\delta},$ implies \eqref{ANScylE2}.
We thus complete the proof of Proposition \ref{propANScylE}.
\end{proof}

\appendix

\section{Proof of \eqref{4.16} and \eqref{4.24}} \label{appendixB}

The goal of this subsection is to present the proof of \eqref{4.16} and \eqref{4.24}. Indeed
in view of the fourth equations of \eqref{eqtuk} and \eqref{eqtvk}, we get, by using integration by parts, that
\begin{itemize}
\item{When $k_1+k_2=k$}
\end{itemize}
\begin{align*}
\int_{\wt\Omega}(\wt u_{k_1}\cdot \wt \nabla +\f{k_2N}r \vt_{k_1})f \cdot g \,r drdz
&=\int_{\wt\Omega} \bigl(-(\wt u_{k_1}\cdot \wt \nabla g) \cdot f +\f{(k_1+k_2)N}r \vt_{k_1}f \cdot g\bigr) \,r drdz\\
&=\int_{\wt\Omega} \bigl(-(\wt u_{k_1}\cdot \wt \nabla g) \cdot f +\f{kN}r \vt_{k_1}f \cdot g\bigr) \,r drdz,
\end{align*}
and
\begin{align*}
\int_{\wt\Omega} (\wt v_{k_1}\cdot \wt \nabla -\f{k_2N}r \ut_{k_1})f \cdot g \,r drdz
&=-\int_{\wt\Omega}\bigl((\wt v_{k_1}\cdot \wt \nabla g) \cdot f +\f{(k_1+k_2)N}r \ut_{k_1}f \cdot g\bigr) \,r drdz\\
&=-\int_{\wt\Omega} \bigl((\wt v_{k_1}\cdot \wt \nabla g) \cdot f +\f{kN}r \ut_{k_1}f \cdot g\bigr) \,r drdz.
\end{align*}

\begin{itemize}
\item{When $k_1-k_2=\pm k$}
\end{itemize}
\begin{align*}
\int_{\wt\Omega} (\wt u_{k_1}\cdot \wt \nabla -\f{k_2N}r \vt_{k_1})f \cdot g \,r drdz
&=\int_{\wt\Omega} \bigl(-(\wt u_{k_1}\cdot \wt \nabla g) \cdot f +\f{(k_1-k_2)N}r \vt_{k_1}f \cdot g\bigr) \,r drdz\\
&=\int_{\wt\Omega} \bigl(-(\wt u_{k_1}\cdot \wt \nabla g) \cdot f \pm\f{kN}r \vt_{k_1}f \cdot g\bigr) \,r drdz,
\end{align*}
and
\begin{align*}
\int_{\wt\Omega} (\wt v_{k_1}\cdot \wt \nabla +\f{k_2N}r \ut_{k_1})f \cdot g \,r drdz
&=-\int_{\wt\Omega} \bigl((\wt v_{k_1}\cdot \wt \nabla g) \cdot f +\f{(k_1-k_2)N}r \ut_{k_1}f \cdot g\bigr) \,r drdz\\
&=-\int_{\wt\Omega} \bigl((\wt v_{k_1}\cdot \wt \nabla g) \cdot f \pm\f{kN}r \ut_{k_1}f \cdot g\bigr) \,r drdz.
\end{align*}

 It is obvious to observe from \eqref{eqtukq} and \eqref{eqtvkq} that we can write $\bigl( (F_k,G_k) \big|\vv u_k\bigr)$ (resp. $\bigl( (F_k,G_k) \big|\pa_z^2\vv u_k\bigr)$) as  a linear combination of the above terms  with $f=\vv u_{k_2}$ and $g=\vv u_k$ (resp. $g=\pa_z^2\vv u_k$), and some lower order terms of the form:
$$\int_{\wt\Omega}\f{\vv u_{k_1}\vv u_{k_2}}r\vv u_k\,rdrdz\quad\Bigl(\text{resp. }\int_{\wt\Omega}\f{\vv u_{k_1}\vv u_{k_2}}r\vv \pa_z^2u_k\,rdrdz\Bigr).$$
Then
\eqref{4.16} follows immediately. The proof of \eqref{4.24} needs an additional integration by parts to move one vertical derivative
 from $\pa_z^2\vv u_k$ to such terms as  $\vv u_{k_1}$ or $\vv u_{k_2}$. We skip the details here.

\section{Estimates of anisotropic Stokes system}\label{appendixA}

In this section, we consider the following anisotropic Stokes system:
\beq\label{STS}
\left\{
\begin{aligned}
&\p_t w-(\Delta_\h+\nu^2\pa_3^2)w+\nabla\mathcal{ P}= f,\quad (t,x)\in R^+\times\Omega,\\
&\dive w=0,\\
&w|_{\p\Omega}=0,
\quad w|_{t=0} =w_{\rm in},
\end{aligned}
\right. \eeq
where $\nu\geq0$ is a constant, the source term $f$ in cylindrical coordinates reads $f=f^r \vv e_r +f^\th \vv e_\th +f^z \vv e_z,$ which can be expanded in the
following Fourier series:
$$
f^\lozenge(t,x)=f^\lozenge_{0}(t,r,z)
+\sum_{k=1}^\infty\bigl(f^\lozenge_{k}(t,r,z)\cos k\th
+g^\lozenge_{k}(t,r,z)\sin k\th\bigr),$$
where $\lozenge$ can be either $r,~\th$ or $z$.
We also write the solution of \eqref{STS} $w=w^r\vv e_r+w^\theta\vv e_\th+\wz\vv e_z$  with $w^\lozenge$ being
 expanded  as
$$w^\lozenge(t,x)=w^\lozenge_{0}(t,r,z)
+\sum_{k=1}^\infty\bigl(
w^\lozenge_{k}(t,r,z)\cos k\th+\varpi^\lozenge_{k}(t,r,z)\sin k\th\bigr).$$
Correspondingly, we also write the scalar pressure function $\mathcal{P}$ as
$$\mathcal{P}(t,x)=\mathcal{P}_{0}(t,r,z)
+\sum_{k=1}^\infty\bigl(\mathcal{P}_{k}(t,r,z)\cos k\th+\mathcal{Q}_{k}(t,r,z)\sin k\th\bigr).$$

By substituting the above equalities into \eqref{STS} and comparing the Fourier coefficients of the resulting equations, we find
\begin{itemize}
\item
$(w^r_0,~\wwt_0,~\wz_0,~\mathcal{P}_0)$ satisfies
\end{itemize}
\begin{equation}\label{eqA.3}
\left\{
\begin{split}
&\p_t\wr_{0}
-\bigl(\p_r^2+\f{\p_r}r+\nu^2\p_z^2
-\f{1}{r^2}\bigr)\wr_{0}+\pa_r \mathcal{P}_{0}
=\fr_0,\\
&\p_t\wwt_{0}
-\bigl(\p_r^2+\f{\p_r}r+\nu^2\p_z^2
-\frac{1}{r^2}\bigr)\wwt_{0}
=\ft_0,\\
&\p_t\wz_{0}
-\bigl(\p_r^2+\f{\p_r}r+\nu^2\p_z^2\bigr)\wz_{0}
+\pa_z \mathcal{P}_{0}
=\fz_0,\\
& \pa_r\wr_{0}+\frac{\wr_{0}}r+\pa_z\wz_{0}=0,\quad (t,r,z)\in\R^+\times\wt\Omega,\\
&(\wr_0,\wwt_0,\wz_0)|_{r=1}=0,
\quad(\wr_0,\wwt_0,\wz_0)|_{t=0}
=(\wr_{{\rm in},0},\wwt_{{\rm in},0},\wz_{{\rm in},0}).
\end{split}
\right.
\end{equation}

\begin{itemize}
\item
 $\left(\wr_{k},~\vpt_{k},~\wz_{k},~\mathcal{P}_k\right)$ for $k\in\N^+$ satisfies
 \end{itemize}
\begin{equation}\label{eqA.4}
\left\{
\begin{split}
&\p_t\wr_{k}
-\bigl(\p_r^2+\f{\p_r}r+\nu^2\p_z^2
-\frac{1+k^2}{r^2}\bigr)\wr_{k}
+\f{2k}{r^2}\vpt_{k}+\pa_r \mathcal{P}_{k} =\fr_k,\\
&\p_t\vpt_{k}
-\bigl(\p_r^2+\f{\p_r}r+\nu^2\p_z^2
-\frac{1+k^2}{r^2}\bigr)\vpt_{k}
+\f{2k}{r^2}\wr_{k}-\f{k}r \mathcal{P}_{k} =\gt_k,\\
&\p_t\wz_{k}
-\bigl(\p_r^2+\f{\p_r}r+\nu^2\p_z^2
-\frac{k^2}{r^2}\bigr)\wz_{k}
+\pa_z \mathcal{P}_{k} =\fz_k,\\
& \pa_r\wr_{k}+\frac{\wr_k}r+\pa_z\wz_{k}
+k\f{\vpt_{k}}r=0,\quad (t,r,z)\in\R^+\times\wt\Omega,\\
&(\wr_k,\vpt_k,\wz_k)|_{r=1}=0,
\quad(\wr_k,\vpt_k,\wz_k)|_{t=0}
=(\wr_{{\rm in},k},\vpt_{{\rm in},k},\wz_{{\rm in},k}).
\end{split}
\right.
\end{equation}

\begin{itemize}
\item
 $\left(\vpr_{k},~\wwt_{k},~\vpz_{k},~\mathcal{Q}_k\right)$ for $k\in\N^+$ satisfies
 \end{itemize}
\begin{equation}\label{eqA.5}
\left\{
\begin{split}
&\p_t\vpr_{k}
-\bigl(\p_r^2+\f{\p_r}r+\nu^2\p_z^2
-\frac{1+k^2}{r^2}\bigr)\vpr_{k}
-\f{2k}{r^2}\wwt_{k}+\pa_r \mathcal{Q}_{k} =\gr_k,\\
&\p_t\wwt_{k}
-\bigl(\p_r^2+\f{\p_r}r+\nu^2\p_z^2
-\frac{1+k^2}{r^2}\bigr)\wwt_{k}
-\f{2k}{r^2}\vpr_{k}+\f{k}r \mathcal{Q}_{k} =\ft_k,\\
&\p_t\vpz_{k}
-\bigl(\p_r^2+\f{\p_r}r+\nu^2\p_z^2
-\frac{k^2}{r^2}\bigr)\vpz_{k}
+\pa_z \mathcal{Q}_{k} =\gz_k,\\
& \pa_r\vpr_{k}+\frac{\vpr_k}r+\pa_z\vpz_{k}
-k\f{\wwt_{k}}r=0,\quad (t,r,z)\in\R^+\times\wt\Omega,\\
&(\vpr_k,\wwt_k,\vpz_k)|_{r=1}=0,
\quad(\vpr_k,\wwt_k,\vpz_k)|_{t=0}
=(\vpr_{{\rm in},k},\wwt_{{\rm in},k},\vpz_{{\rm in},k}).
\end{split}
\right.
\end{equation}

Let us denote
\begin{align*}
&\vv w_0\eqdef(\wr_0,\wwt_0,\wz_0),
\quad\vv w_k\eqdef(\wr_k,\vpt_k,\wz_k,\vpr_k,\wwt_k,\vpz_k),\\
&\vv f_0\eqdef (\fr_0,\ft_0,\fz_0),
\quad\vv f_k\eqdef (\fr_k,\gt_k,\fz_k,\gr_k,\ft_k,\gz_k).
\end{align*}

We shall prove the following
estimates concerning the solutions of \eqref{eqA.3}-\eqref{eqA.5}:

\begin{prop}\label{propA}
{\sl
Let $w_{\rm in}\in L^2$ and $f\in L^1(\R^+;L^2)$. Then for any $t>0$, we have
\begin{equation}\begin{split}\label{eqA.6}
\|\vv w_0 \|_{L^\infty_t(L^2)}^2
+\|(\pa_r,\nu\pa_z) \vv w_0\|_{L^2_t(L^2)}^2
+\bigl\|\f{(\wr_0,\wwt_0)}r\bigr\|_{L^2_t(L^2)}^2
\leq \|\vv w_{{\rm in},0}\|_{L^2}^2
+2\bigl(\vv f_0|\vv w_0\bigr)_{L^2_t(L^2)},
\end{split}\end{equation}
and for $k\geq 1$
\begin{equation}\label{eqA.7}
\|\vv w_k \|_{L^\infty_t(L^2)}^2
+\|(\pa_r,\nu\pa_z)\vv w_k\|_{L^2_t(L^2)}^2
+(k-1)^2\bigl\|\f{\vv w_k}r\bigr\|_{L^2_t(L^2)}^2 \leq \|\vv w_{{\rm in},k}\|_{L^2}^2
+2\bigl(\vv f_k|\vv w_k\bigr)_{L^2_t(L^2)}.
\end{equation}
}\end{prop}

\begin{proof} For simplicity, we just present the {\it a priori} estimates for smooth enough solutions of \eqref{eqA.3}-\eqref{eqA.5}.
Due to the homogeneous Dirichlet boundary conditions for $\vv w_0, \vv w_k,$  and the fourth equation of \eqref{eqA.3}-\eqref{eqA.5}, we get, by using integration by parts, that
\begin{align*}
&\int_{\wt\Omega}\bigl( \wr_0 \p_r \mathcal{P}_0
+\wz_0 \p_z \mathcal{P}_0\bigr)\,rdrdz
=-\int_{\wt\Omega}\bigl( \pa_r\wr_0+\frac{\wr_0}r
+\pa_z\wz_0\bigr)\mathcal{P}_0\,rdrdz=0,\\
&
\int_{\wt\Omega}\bigl(\wr_k\p_r
\mathcal{P}_k-\f{k}r\vpt_k \mathcal{P}_k
+\wz_k \p_z \mathcal{P}_k\bigr)\,rdrdz
=-\int_{\wt\Omega}\bigl( \pa_r\wr_{k}+\frac{\wr_k}r+\f{k}r \vpt_k+\pa_z\wz_{k}\bigr)\mathcal{P}_k\,rdrdz=0,\\
&
\int_{\wt\Omega}\bigl( \vpr_k \p_r \mathcal{Q}_k+\f{k}r \wwt_k \mathcal{Q}_k+\vpz_k \p_z\mathcal{Q}_k\bigr)\,rdrdz=-\int_{\wt\Omega}\bigl( \pa_r\vpr_{k}+\frac{\vpr_k}r-\f{k}r \wwt_k+\pa_z\vpz_{k}\bigr)\mathcal{Q}_k\,rdrdz=0.
\end{align*}
So that by taking $L^2$ inner product of \eqref{eqA.3} with $\vv w_0$ and using integration by parts, we find
\begin{equation}\label{eqA.8}
\f12\f{d}{dt}\|\vv w_0 \|_{L^2}^2
+\|\pa_r \vv w_0\|_{L^2}^2
+\nu^2\|\pa_z \vv w_0\|_{L^2}^2
+\bigl\|\f{(\wr_0,\wwt_0)}r\bigr\|_{L^2}^2 =\bigl(\vv f_0 |\vv w_0\bigr)_{L^2}.
\end{equation}
Whereas by taking $L^2$ inner product of \eqref{eqA.4} with $(\wr_k,\vpt_k,\wz_k)$   and \eqref{eqA.5} with $(\vpr_k,\wwt_k,\vpz_k)$, we obtain
\begin{equation}\label{eqA.9}
\begin{aligned}
\f12\f{d}{dt}\|\vv w_k\|_{L^2}^2
+\|\pa_r \vv w_k\|_{L^2}^2
&+\nu^2\|\pa_z \vv w_k\|_{L^2}^2
+k^2\bigl\|\f{\vv w_k}r\bigr\|_{L^2}^2
+\bigl\|\f{(\wr_k,\wwt_k,\vpt_k,\vpr_k)}r\bigr\|_{L^2}^2\\
&\qquad\qquad=4k\bigl(\f{\vpr_k}r |\f{\wwt_k}r\bigr)_{L^2}
-4k\bigl(\f{\wr_k}r |\f{\vpt_k}r\bigr)_{L^2}
+\bigl(\vv f_k |\vv w_k\bigr)_{L^2}.
\end{aligned}\end{equation}
Observing that
$$
4k\bigl|\bigl(\f{\vpr_k}r |\f{\wwt_k}r\bigr)_{L^2}
-\bigl(\f{\wr_k}r |\f{\vpt_k}r\bigr)_{L^2}\bigr|
\leq 2k\bigl\|\f{(\wr_k,\wwt_k,\vpt_k,\vpr_k)}r
\bigr\|_{L^2}^2.
$$
By inserting the above inequality into \eqref{eqA.9}, we arrive at
\begin{equation}\label{eqA.10}
\f12\f{d}{dt}\|\vv w_k\|_{L^2}^2
+\|\pa_r\vv w_k\|_{L^2}^2
+\nu^2\|\pa_z\vv w_k\|_{L^2}^2
+(k-1)^2\bigl\|\f{\vv w_k}r\bigr\|_{L^2}^2
\leq\bigl(\vv f_k |\vv w_k\bigr)_{L^2}.
\end{equation}

By integrating \eqref{eqA.8} and \eqref{eqA.10} over $[0,t],$ we achieve \eqref{eqA.6} and \eqref{eqA.7}. This completes the proof of  Proposition
\ref{propA}.
\end{proof}

As a direct consequence of Proposition \ref{propA}, one has
\begin{cor}\label{corA}
{\sl
For any given $k_0\in\N$, if $\vv f_k=\vv w_{{\rm in},k}=0$ for all $k\neq k_0$, then the solution $w$ to \eqref{STS}  contains only $k_0-$th Fourier mode in the $\th$ variable, i.e. $\vv w_k=0$ for all $k\neq k_0$.
}
\end{cor}

\noindent\textbf{Acknowledgments}
Y. Liu is supported by NSF of China under grant 12101053.
Ping Zhang is supported by National Key R$\&$D Program of China under grant
  2021YFA1000800 and K. C. Wong Education Foundation.
 He is also partially supported by National Natural Science Foundation of China under Grants  12288201 and 12031006.\medskip

\noindent\textbf{Declarations}\medskip

\noindent\textbf{Conflict of interest} The authors declared that they have no conflicts of interest to this work.\medskip

\noindent\textbf{Publisher’s Note} Springer Nature remains neutral with regard to jurisdictional claims
in published maps and institutional affiliations.
\medskip

\noindent Springer Nature or its licensor (e.g. a society or other partner) holds exclusive rights
to this article under a publishing agreement with the author(s) or other rightsholder(s);
author self-archiving of the accepted manuscript version of this article is solely governed
by the terms of such publishing agreement and applicable law.

\end{document}